\documentclass{amsart}
\usepackage{amsmath,amssymb,amsthm}
\usepackage{booktabs}
\usepackage[usenames,dvipsnames]{color}
\usepackage{comment}
\usepackage{graphicx}
\usepackage{latexsym}
\usepackage[latin1]{inputenc}
\usepackage{mathtools}
\usepackage[shortlabels]{enumitem}
\usepackage{thmtools}
\usepackage{url}
\usepackage[all,cmtip]{xy}
\usepackage[pagebackref]{hyperref} 

\title{Structure and enumeration of $(\mathbf{3}+\mathbf{1})$-free posets}

\author{Mathieu Guay-Paquet}
\address{
LaCIM \\
Universit\'e du Qu\'ebec \`a Montr\'eal \\
201 Pr\'esident-Kennedy \\
Montr\'eal QC\ \ H2X~3Y7 \\
Canada}
\thanks{MGP was supported by an NSERC Postdoctoral Fellowship}
\email{mathieu.guaypaquet@lacim.ca}

\author{Alejandro H. Morales}
\thanks{AHM was supported by a CRM-ISM Postdoctoral Fellowship}
\email{ahmorales@lacim.ca}

\author{Eric Rowland}

\definecolor{darkgreen}{rgb}{0,0.7,0}
\definecolor{purplish}{rgb}{0.5,0,0.8}

\hypersetup{
  breaklinks,colorlinks,citecolor=darkgreen,linkcolor=purplish, 
  pdftitle={Structure and enumeration of (3+1)-free posets},
  pdfauthor={Mathieu Guay-Paquet, Alejandro H. Morales, and Eric Rowland},
}

\newcommand{\NN}{\mathbb{N}}
\newcommand{\id}{\mathrm{id}}

\newcommand{\bl}{b_{\textnormal{lbl}}}
\newcommand{\bu}{b_{\textnormal{unl}}}
\newcommand{\Bl}{B_{\textnormal{lbl}}}
\newcommand{\Bu}{B_{\textnormal{unl}}}
\newcommand{\pl}{p_{\textnormal{lbl}}}
\newcommand{\pu}{p_{\textnormal{unl}}}
\newcommand{\Hl}{H_{\textnormal{lbl}}}
\newcommand{\Hu}{H_{\textnormal{unl}}}

\newcommand{\Cl}{C_{\textnormal{lbl}}}
\newcommand{\Cu}{C_{\textnormal{unl}}}
\newcommand{\Tl}{T_{\textnormal{lbl}}}
\newcommand{\Tu}{T_{\textnormal{unl}}}
\newcommand{\Pl}{P_{\textnormal{lbl}}}
\newcommand{\Pu}{P_{\textnormal{unl}}}

\newcommand{\skel}{\mathcal{S}}

\newcommand{\aut}{\mathop{\mathrm{Aut}}}
\newcommand{\abs}[1]{\left|{#1}\right|}

\newcommand{\tpo}{\texorpdfstring{$(\mathbf{3}+\mathbf{1})$}{(3+1)}}
\newcommand{\tpt}{\texorpdfstring{$(\mathbf{2}+\mathbf{2})$}{(2+2)}}
\newcommand{\pathto}{\twoheadrightarrow}

\declaretheorem[numberwithin=section]{theorem}
\declaretheorem[numberlike=theorem]{lemma}
\declaretheorem[numberlike=theorem]{proposition}

\declaretheorem[numberlike=theorem, style=definition]{definition}
\declaretheorem[numberlike=theorem, style=definition]{remark}
\declaretheorem[numberlike=theorem, style=definition]{example}
\declaretheorem[numberlike=theorem, style=definition]{construction}
\declaretheorem[numbered=no, style=definition]{outline}

\numberwithin{equation}{section} 

\begin{document}

\begin{abstract}
A poset is \tpo-free if it does not contain the disjoint union of chains of length 3 and 1 as an induced subposet.
These posets play a central role in the \tpo-free conjecture of Stanley and Stembridge.
Lewis and Zhang have enumerated \tpo-free posets in the graded case by decomposing them into bipartite graphs, but until now the general enumeration problem has remained open.
We give a finer decomposition into bipartite graphs which applies to all \tpo-free posets and obtain generating functions which count \tpo-free posets with labelled or unlabelled vertices.
Using this decomposition, we obtain a decomposition of the automorphism
group and asymptotics for the number of \tpo-free posets.
\end{abstract}

\maketitle

\section{Introduction}\label{sec:intro}

A poset $P$ is \emph{$(\mathbf{i}+\mathbf{j})$-free} if it contains no induced subposet that is isomorphic to the poset consisting of two disjoint chains of lengths $i$ and $j$.
In particular, $P$ is \tpo-free if there are no vertices $x, y, z, w \in P$ such that $x < y < z$ and $w$ is incomparable to $x$, $y$, and $z$.

Posets that are \tpo-free play a role in the study of
Stanley's chromatic symmetric function~\cite{St1, St2}, a symmetric
function associated with a poset that generalizes the chromatic
polynomial of a graph.
Namely, a well-known conjecture of Stanley and Stembridge~\cite{StSt} is that the chromatic symmetric function of a \tpo-free poset has positive coefficients in the basis of elementary symmetric functions. As evidence toward this conjecture, Stanley~\cite{St1} verified the conjecture for the class of $\mathbf{3}$-free posets, and Gasharov~\cite{G2} has shown the weaker result that the chromatic symmetric function of a \tpo-free poset is Schur-positive.

To make more progress toward the Stanley--Stembridge conjecture, a
better understanding of \tpo-free posets is needed. Skandera and Reed~\cite{Sk1, SkR} have given structural results and a
characterization of \tpo-free posets in terms of their antiadjacency matrix.
In addition, certain families of \tpo-free posets have been
enumerated. For example, the number of \tpo-and-\tpt-free
posets with $n$ vertices is the $n$th Catalan number~\cite[Ex.~6.19(ddd)]{EC2};
Atkinson, Sagan and Vatter~\cite{ASV} have enumerated the permutations
that avoid the patterns $2341$ and $4123$, which give rise to the
\tpo-free posets of dimension two; and Lewis and
Zhang~\cite{LZ} have made significant progress by enumerating
\emph{strongly graded} \tpo-free posets in terms of bicoloured graphs
\footnote{We use the term \emph{bicoloured} rather than the term \emph{bipartite} to emphasize the fact that a 2-colouring is not only possible, but actually given and fixed; in particular, there is only one bipartite graph with one vertex, but there are two bicoloured graphs with one vertex.}
using a new structural decomposition.
However, until now the general enumeration problem for \tpo-free posets has remained open~\cite[Ex.~3.16(b)]{EC1}.

In this paper, we give generating functions for
\tpo-free posets with unlabelled and labelled vertices in terms
of the generating functions for bicoloured graphs with unlabelled and
labelled vertices, respectively. As in the strongly graded case, the two problems are equally hard, although the enumeration problem for bicoloured graphs has received more attention.

In the unlabelled case, let $\pu(n)$ be the number of \tpo-free posets with $n$ unlabelled vertices, and let $S(c,t)$ be the unique formal power series solution (in $c$ and $t$) of the cubic equation
\begin{equation}\label{recurrTandC}
  S(c,t) = 1 + \frac{c}{1+c}S(c,t)^2 + tS(c,t)^3.
\end{equation}
We show that the ordinary generating function for unlabelled \tpo-free posets is
\begin{equation}\label{ordgs}
  \sum_{n\geq 0} \pu(n) x^n
    = S\big(x/(1-x), 1-2x-\Bu(x)^{-1}\big),
\end{equation}
where $\Bu(x) = 1 + 2x + 4x^2 + 8x^3 + 17x^4 + \cdots$ is the ordinary generating function for unlabelled bicoloured graphs. Before our investigation, the On-Line Encyclopedia of Integer Sequences~\cite{OEIS} had 22 terms in the entry~\cite[\href{http://oeis.org/A049312}{A049312}]{OEIS} for the coefficients of $\Bu(x)$, but only 7 terms in the entry~\cite[\href{http://oeis.org/A079146}{A079146}]{OEIS} for the numbers $\pu(n)$. Using~\eqref{ordgs}, we have closed this gap; the numbers $\pu(n)$ for $n = 0, 1, 2, \ldots, 22$ are
\begin{quote}
1, 1, 2, 5, 15, 49, 173, 639, 2469, 9997, 43109, 205092, 1153646,
8523086, 91156133, 1446766659, 32998508358, 1047766596136,
\linebreak
45632564217917, 2711308588849394, 219364550983697100,
\linebreak
24151476334929009951, 3618445112608409433287.
\end{quote}

Similarly, in the labelled case, let $\pl(n)$ be the number of \tpo-free posets with $n$ labelled vertices. We show that the exponential generating function for labelled \tpo-free posets is
\begin{equation}\label{expgs}
  \sum_{n\geq 0} \pl(n) \frac{x^n}{n!}
    = S\big(e^x-1, 2e^{-x} -1 - \Bl(x)^{-1}\big),
\end{equation}
where $\Bl(x) = \sum_{n\geq 0} \sum_{i=0}^n \binom{n}{i} 2^{i(n-i)} \frac{x^{n}}{n!}$ is the exponential generating function for labelled bicoloured graphs. Such bicoloured graphs are easy to count, but before our investigation the OEIS had only 9 terms in the entry~\cite[\href{http://oeis.org/A079145}{A079145}]{OEIS} for $\pl(n)$. Using~\eqref{expgs}, arbitrarily many terms $\pl(n)$ can be computed.

Our main tool is a new decomposition of \tpo-free posets called the \emph{canonical partition} into blocks called \emph{clone sets} and \emph{tangles}, with the relations between blocks given by a \emph{skeleton}.
This partition is compatible with the automorphism group, in the sense that for a \tpo-free poset $P$, $\aut(P)$ breaks up as the direct product of the automorphism group of each block.
This decomposition also generalizes a decomposition of Skandera
and Reed~\cite{SkR} for \tpo-and-\tpt-free posets given by \emph{altitudes} of vertices.
In terms of generating functions, the restriction of our results to \tpo-and-\tpt-free posets corresponds to the specialization $t=0$ in~\eqref{recurrTandC}. Indeed, one can see that $S(x/(1-x),0)$ satisfies the functional equation for the Catalan generating function, as expected.

\begin{remark}
Using the notions of skeleta, clone sets and tangles, it is possible to quickly generate all \tpo-free posets of a given size up to isomorphism in a straightforward way. With this approach, we were able to list all \tpo-free posets on up to 11 vertices in a few minutes on modest hardware.
Note that this technique can accommodate the generation of interesting
subclasses of \tpo-free posets (\textit{e.g.}, \tpt-free, weakly
graded, strongly graded, co-connected, fixed number of levels) or
constructing these posets from the bottom up, level by level (which
can help compute invariants like the chromatic symmetric function, as
in \cite{GP}).
\end{remark}

\begin{remark}
Comparing the list of numbers above with data provided by Joel Brewster
Lewis for the number of strongly graded \tpo-free posets~\cite[\href{http://oeis.org/A222863}{A222863}, \href{http://oeis.org/A222865}{A222865}]{OEIS}, it appears that, asymptotically,
almost all \tpo-free posets are strongly graded. We prove this in \autoref{sec:asympt}, building
on the asymptotic analysis of Lewis and Zhang for the strongly graded
\tpo-free posets.
In fact, almost all \tpo-free posets are $\mathbf{3}$-free, so their Hasse diagrams are bicoloured graphs.
Since Stanley~\cite{St1} verified the Stanley-Stembridge conjecture for the class of $\mathbf{3}$-free posets it follows that this conjecture is true for almost all \tpo-free posets.
\end{remark}

\begin{outline}
In \autoref{sec:structure}, we carefully construct some equivalent data structures for \tpo-free posets (auxiliary digraphs, fleshed out skeleta) in order to define the canonical partition and decompose the automorphism group for these posets.
In \autoref{sec:enumeration}, we carry out the enumeration of \tpo-free posets by computing generating functions for clone sets, tangles and skeleta in terms of the generating function for bicoloured graphs.
In \autoref{sec:asympt}, we give asymptotics for the number of \tpo-free posets and compare them to other related classes of posets.
Our main theorems are \autoref{thm:bijection}, \autoref{thm:automorphism}, \autoref{thm:enum} and \autoref{thm:asymptP}.
\end{outline}

An extended abstract of this article appeared as \cite{GMRProc}.

\section{Structure}\label{sec:structure}

Our goal is to study the structure of \tpo-free posets, and our strategy will be to take the original definition in terms of an order relation on a set of vertices, and to progressively rephrase it first in terms of an auxiliary digraph structure, then in terms of a canonical partition of the vertex set together with a dependence graph on its blocks. At each step, we carefully define the structures we are using, and show through bijections that they represent the initial \tpo-free posets faithfully.

\subsection{\tpo-free posets}

Let us start a definition and some basic properties.

\begin{definition}[\tpo-free poset]
  A \emph{\tpo-free poset} $P = (V, <)$ consists of a (finite) set $V$ of \emph{vertices}, together with an order relation $<$ on $V$ which does not induce a copy of the \tpo{} poset as a subposet. Equivalently, there does not exist vertices $x, y, z, w \in V$ such that $x < y < z$ with $w$ incomparable to each of $x, y, z$.
\end{definition}

In general, \tpo-free posets are not graded, so we can't speak of the rank of a given vertex in such a poset. However, it will be useful to still have some notion of `how high up' a given vertex is.

\begin{definition}[level function]
  The \emph{level function} for a \tpo-free poset $P = (V, <)$ is the function $\ell \colon V \to \NN$ defined recursively by
  \[
    \ell(x) = \begin{cases}
      0 &\text{if $x$ is a minimal element,} \\
      \max_{y < x} \ell(y) + 1 &\text{otherwise.}
    \end{cases}
  \]
  The sets $\ell^{-1}(\{k\})$ for $k = 0, 1, 2, \ldots$ partition the vertex set into \emph{levels}.
\end{definition}

\begin{figure}[b]
\includegraphics{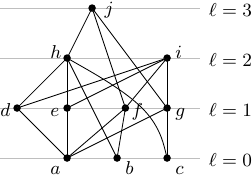}
\qquad \quad
\includegraphics{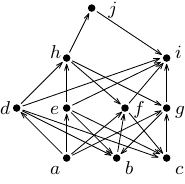}
\caption{Left: the Hasse diagram of a \tpo-free poset $P$ with 10 vertices, with the levels indicated. Right: the constructed auxiliary digraph $A$ (applying \autoref{cons:P-to-A}).}
\label{fig:poset-auxgraph}
\end{figure}

\begin{example}
  The left side of \autoref{fig:poset-auxgraph} shows a \tpo-free poset $P$ on 10 vertices, with levels indicated.
\end{example}

The next proposition characterizes the possible level functions of \tpo-free posets.

\begin{proposition}\label{prop:level-function}
  Given a finite set $V$ of vertices, a function $\ell \colon V \to \NN$ is the level function for \emph{some} \tpo-free poset $P = (V, <)$ iff for every vertex $x \in V$ with $\ell(x) > 0$, there exists a vertex $y \in V$ with $\ell(y) = \ell(x) - 1$.
\end{proposition}

\begin{proof}
  Suppose $\ell \colon V \to \NN$ is the level function for a poset $P = (V, <)$. If $x \in V$ is a vertex with $\ell(x) > 0$, the maximum in the definition of $\ell(x)$ must be reached by some $y \in V$, which satisfies $\ell(y) = \ell(x) - 1$.

  Conversely, if $\ell \colon V \to \NN$ is a function such that for every vertex $x \in V$ with $\ell(x) > 0$, there exists a vertex $y \in V$ with $\ell(y) = \ell(x) - 1$, then we can define a \tpo-free poset $P = (V, <)$ by letting $x < y$ iff $\ell(x) < \ell(y)$. This $P$ is in fact a $(\mathbf{2}+\mathbf{1})$-free poset, since any two vertices are comparable unless they are on the same level, and the given function $\ell$ satisfies the definition of level function for this poset.
\end{proof}

\subsection{Auxiliary digraphs}

Although the order relation $<$ gives us some information about every pair of vertices, the condition of being \tpo-free imposes strong constraints, so that this information is redundant for most pairs of vertices if the level function $\ell$ is known, as shown by the following proposition.

\begin{proposition}
  Let $x, y \in V$ be two vertices of a \tpo-free poset $P = (V, <)$ with level function $\ell$. Then, $x < y$ implies $\ell(y) \geq \ell(x) + 1$, and $\ell(y) \geq \ell(x) + 2$ implies $x < y$.
\end{proposition}

\begin{proof}
  The first implication follows directly from the definition of the level function. For the second implication, suppose $\ell(y) \geq \ell(x) + 2$. The definition of the level function guarantees the existence of vertices $a, b \in V$ with $a < b < y$ and $\ell(b) = \ell(y) - 1$, $\ell(a) = \ell(b) - 1$. Since $P$ is \tpo-free, the vertex $x$ must be comparable to at least one of $a, b, y$, and by level considerations, it cannot be greater than any of these vertices, so we must have $x < y$.
\end{proof}

\begin{remark}
  Lewis and Zhang~\cite[Theorem~3.1]{LZ} have a version of this proposition for strongly graded \tpo-free posets. The proofs are essentially the same.
\end{remark}

\begin{remark}
  Note that the covering relations of $P$ include all relations $x < y$ with $\ell(y) = \ell(x) + 1$, but they may also include relations with $\ell(y) = \ell(x) + 2$.
\end{remark}

\begin{example}
  For the poset $P$ depicted in \autoref{fig:poset-auxgraph}, the relations $b < h$, $c < h$, $f < j$ and $g < j$ are covering relations even though they relate elements which are not on adjacent levels.
\end{example}

To factor out this redundancy, we use the notion of an auxiliary digraph, which only records information about pairs of vertices that are on adjacent levels. There are many possible ways to represent the same information, but we choose digraphs because the language and tools of directed cycles will be useful later.

\begin{definition}[auxiliary digraph]\label{def:auxiliary}
  An \emph{auxiliary digraph} $A = (V, \ell, E)$ consists of a (finite) set $V$ of vertices, together with a level function $\ell$ (in the sense of \autoref{prop:level-function}) on $V$, and a set $E \subseteq V^2$ of directed \emph{edges}, denoted $x \to y$ if $(x, y) \in E$, such that:
  \begin{enumerate}[({A}1)]
    \item\label{item:a-edges} there is an edge between vertices $x$ and $y$ (that is, either $x \to y$ or $y \to x$, but not both) iff the vertices are on adjacent levels (that is, $\ell(y) = \ell(x) \pm 1$);
    \item\label{item:a-initial} if $\ell(x) > 0$, then there exists a vertex $y$ with $\ell(y) = \ell(x) - 1$ and $y \to x$; and
    \item\label{item:a-notall} there are no \emph{tall 4-cycles}, that is, vertices $x, y, z, w \in V$ not contained in a pair of adjacent levels such that $x \to y \to z \to w \to x$.
  \end{enumerate}
\end{definition}

There are many details to check, but the constructions used to translate between the data $P = (V, <)$ of a \tpo-free poset and the data $A = (V, \ell, E)$ of an auxiliary digraph are fairly straightforward.

\begin{construction}\label{cons:P-to-A}
  Given a \tpo-free poset $P = (V, <)$, we can construct an auxiliary digraph $A = (V, \ell, E)$ as follows:
  \begin{enumerate}[i., font=\itshape]
    \item keep the same vertex set $V$;
    \item take $\ell$ to be the level function for $P$; and
    \item for every pair of vertices $x, y \in V$ with $\ell(y) = \ell(x) + 1$, either let $x \to y$ if $x < y$, or let $y \to x$ if $x \not< y$.
  \end{enumerate}
\end{construction}

\begin{example}
  Starting from the poset $P$ on the left side of \autoref{fig:poset-auxgraph}, we can obtain the auxiliary digraph $A$ on the right side by keeping each vertex on the same level; putting a complete bipartite graph between each pair of adjacent levels; and orienting each edge upwards when it appears as a relation in the Hasse diagram (so that the relation $h < j$ becomes the edge $h \to j$), or downwards when it does not appear as a relation in the Hasse diagram (so that the non-relation $i \not< j$ becomes the edge $j \to i$).
\end{example}

\begin{proof}
  Let us check that the construction satisfies the defining properties of an auxiliary digraph.
  Condition~\ref{item:a-edges} clearly holds.
  For Condition~\ref{item:a-initial}, for any vertex $x \in V$ with $\ell(x) > 0$, the definition of $\ell(x)$ guarantees the existence of a $y < x$ with $\ell(y) = \ell(x) - 1$, and for this choice of $y$ we have $y \to x$.
  To check Condition~\ref{item:a-notall}, note that any four vertices which form a tall 4-cycle in the construction of $A$ would necessarily form an induced copy of the \tpo{} poset in the original poset $P$, so this is ruled out.
\end{proof}

\begin{construction}\label{cons:A-to-P}
  Given an auxiliary digraph $A = (V, \ell, E)$, we can construct a \tpo-free poset $P = (V, <)$ as follows:
  \begin{enumerate}[i., font=\itshape]
    \item keep the same vertex set $V$; and
    \item for every pair of vertices $x, y \in V$, let $x < y$ iff either $\ell(y) = \ell(x) + 1$ and $x \to y$, or $\ell(y) \geq \ell(x) + 2$.
  \end{enumerate}
\end{construction}

\begin{example}
  To reverse the construction, starting from the auxiliary digraph $A$ on the right side of \autoref{fig:poset-auxgraph}, we can obtain the \tpo-free poset $P$ on the left side by putting in all relations for vertices that are at least two levels apart (such as $b < h$, $b < i$, $b < j$); and putting in relations for every edge which is oriented upwards (so that $h \to j$ becomes $h < j$, but $j \to i$ is ignored and $i \not< j$).
\end{example}

\begin{proof}
  Let us check that the constructed $P$ is indeed a \tpo-free poset.
  Since $x < y$ implies $\ell(y) \geq \ell(x) + 1$, and $\ell(y) \geq \ell(x) + 2$ implies $x < y$, the constructed relation $<$ is irreflexive and transitive, so it does define a poset structure.
  To check \tpo-freedom, suppose $w \in V$ is a vertex and the vertices $x, y, z \in V$ are incomparable to it. Then, we must have $\{\ell(x), \ell(y), \ell(z)\} \subseteq \{\ell(w)-1, \ell(w), \ell(w)+1\}$ and we could only have $x < y < z$ if the vertices $x, y, z, w$ formed a tall 4-cycle in the original auxiliary digraph $A$. This is ruled out by definition, so the poset $P$ is \tpo-free.
\end{proof}

\begin{proposition}\label{prop:bijection1}
  \autoref{cons:P-to-A} and \autoref{cons:A-to-P} are inverses of each other. Hence, they establish a bijection between \tpo-free posets and auxiliary digraphs.
\end{proposition}

\begin{proof}
  For the first direction (applying \autoref{cons:P-to-A} first, then \autoref{cons:A-to-P}), we need to show that the relation $<$ is preserved.

  Let $P_0 = (V, <)$ be the original \tpo-free poset, $A = (V, \ell, E)$ the constructed auxiliary digraph, and $P_1 = (V, \prec)$ the constructed \tpo-free poset. Let $x, y \in V$ be two vertices, and without loss of generality, assume $\ell(y) \geq \ell(x)$. If $\ell(y) = \ell(x)$, then $x$ and $y$ are incomparable in $P_0$ and in $P_1$. If $\ell(y) \geq \ell(x) + 2$, then $x < y$ and $x \prec y$. If $\ell(y) = \ell(x) + 1$ and $x < y$, then $x \to y$ in the auxiliary digraph $A$, and $x \prec y$. The only remaining case is $\ell(y) = \ell(x) + 1$ with $x$ and $y$ incomparable in $P_0$; then, $x \not\to y$ in $A$, and $x$ and $y$ are incomparable in $P_1$ as well. Thus, $P_0 = P_1$.

  For the other direction (applying \autoref{cons:A-to-P} first, then \autoref{cons:P-to-A}), we need to show that the level function $\ell$ is preserved, and that the edges $x \to y$ are preserved.

  Let $A_0 = (V, \ell_0, E_0)$ be the original auxiliary digraph, $P = (V, <)$ be the constructed \tpo-free poset, and $A_1 = (V, \ell_1, E_1)$ be the constructed auxiliary digraph. It follows from Condition~\ref{item:a-initial} in the definition of auxiliary digraphs that $\ell_0$ satisfies the defining relation for the level function of $P$, so we have $\ell_0 = \ell_1$ as functions. Let $x, y \in V$ be two vertices. There is an edge in $A_0$ (and in $A_1$) between them iff $\ell_0(y) = \ell_0(x) \pm 1$; without loss of generality, assume $\ell_0(y) = \ell_0(x) + 1$. If $x \to y$ in $A_0$, then $x < y$ in $P$, and $x \to y$ in $A_1$ as well. Otherwise, if $y \to x$ in $A_0$, then $x \not< y$ in $P$, and $y \to x$ in $A_1$ as well. Thus, we have $E_0 = E_1$, and $A_0 = A_1$.
\end{proof}

\subsection{Cycle lemmas}\label{sec:cycles}

Our next task is to define the canonical partition of the vertex set $V$. To do this, we need some facts about the strongly connected components of the auxiliary digraph, so we turn our attention to its cycles.

\begin{definition}[tall cycle, squat cycle]
  As in \autoref{def:auxiliary}, we say that a (directed) cycle in the auxiliary digraph is \emph{tall} if it contains vertices from at least three different levels.
Conversely, we say that a cycle is \emph{squat} if it is contained in a pair of adjacent levels.
\end{definition}

\begin{figure}[b]
\[
\vcenter{\hbox{\fbox{\includegraphics{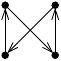}}}}
\hspace{2em}
\vcenter{\hbox{\fbox{\includegraphics{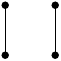}}}}
\hspace{3.5em}
\vcenter{\hbox{\fbox{\includegraphics{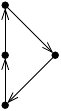}}}}
\hspace{2em}
\vcenter{\hbox{\fbox{\includegraphics{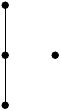}}}}
\]
\caption{From left to right:
a squat 4-cycle;
the corresponding induced subposet, which is a copy of \tpt{};
a tall 4-cycle;
the corresponding induced subposet, which is a copy of \tpo{}.}
\label{fig:short-cycles}
\end{figure}

\begin{remark}\label{rem:tpt}
  As noted earlier in the proofs of \autoref{cons:P-to-A} and \autoref{cons:A-to-P}, a tall 4-cycle corresponds to an induced copy of the \tpo{} poset. A squat 4-cycle corresponds to an induced copy of the \tpt{} poset. (See \autoref{fig:short-cycles}.) Since the auxiliary digraph is bipartite, these are the shortest possible cycles.
\end{remark}

The following proposition gives substantial restrictions on what the strongly connected components of the auxiliary digraph can look like, and shows that they can be computed by looking only at the squat 4-cycles. We will prove it as a series of lemmas.

\begin{proposition}\label{prop:tpt}
  Each non-trivial strongly connected component of the auxiliary digraph is contained in a pair of adjacent levels, and is generated by squat 4-cycles.
\end{proposition}

\begin{proof}
  For vertices $x, y \in V$, the relation `$x$ and $y$ are in the same strongly connected component' is the transitive closure of the relation `$x$ and $y$ are in a (directed) cycle', so we can prove the proposition by looking only at the cycles of the auxiliary digraph. By \autoref{lem:no-tall}, every cycle is contained in a pair of adjacent levels, and by \autoref{lem:no-squat-union}, this also holds for every connected union of squat 4-cycles. Each non-trivial strongly connected component can be obtained as a connected union of cycles, and by \autoref{lem:squat-gen}, each of these in turn is a connected union of squat 4-cycles, which completes the proof.
\end{proof}

\begin{lemma}\label{lem:no-tall}
  There are no tall cycles in the auxiliary digraph.
\end{lemma}

\begin{proof}
  We proceed by contradiction, and consider a \emph{shortest} tall cycle $C$. Let $x_1$ be a vertex of $C$ on its lowest level. Since $C$ is a tall cycle, if we follow it starting at $x_1$, we must eventually reach a vertex $z$ which is two levels higher. Let
\[
\xymatrix@C=0.5em@R=0.5em{
            &           &             &            &       &                             &                            &            &           &z \\
y_0 \ar[dr] &           & y_1 \ar[dr] &            &y_2    &\cdots                       &y_{k\mathrlap{{}-1}}\ar[dr] &            &y_k\ar[ur] &  \\
            &x_1 \ar[ur]&             &x_2 \ar[ur] &\cdots &x_{k\mathrlap{{}-1}} \ar[ur] &                            &x_k \ar[ur] &           &
}
\]
be the segment of $C$ from the predecessor $y_0$ of $x_1$ to $z$, where each $x_i$ is on the lowest level of $C$, and each $y_i$ is on the level above. Since they are on adjacent levels, the vertices $y_{k-1}$ and $z$ must be joined by an edge in the auxiliary digraph; however, the auxiliary digraph does not contain tall 4-cycles, and we already have the edges $y_{k-1} \to x_k \to y_k \to z$, so we must have $y_{k-1} \to z$. Let $D$ be the cycle obtained from $C$ by taking a shortcut along the edge $y_{k-1} \to z$:
\[
\xymatrix@C=0.5em@R=0.5em{
              &            &           &             &            &       &                             &                              & & &z \ar[r] &\cdots \\
\cdots \ar[r] &y_0 \ar[dr] &           & y_1 \ar[dr] &            &y_2    &\cdots                       &y_{k\mathrlap{{}-1}}\ar[urrr] & &\mbox{\phantom{$y_k$}} &         &      \\
              &            &x_1 \ar[ur]&             &x_2 \ar[ur] &\cdots &x_{k\mathrlap{{}-1}} \ar[ur] &                              &\mbox{\phantom{$x_k$}} & &         &
}
\]
By assumption, $C$ is a shortest tall cycle, so the shorter cycle $D$ must be squat. In particular, since it contains the vertices $y_{k-1}$ and $z$, $D$ cannot contain any of the $x_i$ vertices on the level below, or any vertices on a higher level. Thus, we have $k-1 = 0$, and $C$ has a unique vertex on its lowest level, namely $x_1$. By the same argument, turned upside-down, $C$ also has a unique vertex on its highest level, namely $z$. But then, the cycle $C$ is simply
\[
\xymatrix@C=0.5em@R=0.5em{
            &            &            &z \ar[dlll] \\
y_0 \ar[dr] &            &y_1 \ar[ur] &            \\
            &x_1 \ar[ur] &            &
}
\]
which is a tall 4-cycle, and a contradiction.
\end{proof}

\begin{lemma}\label{lem:squat-gen}
  Each squat cycle in the auxiliary digraph is generated by squat 4-cycles.
\end{lemma}

\begin{proof}
Let $C$ be the given squat cycle. We proceed by induction on the length of $C$. If every edge of $C$ is contained in a squat 4-cycle formed from vertices of $C$, then we are done. Otherwise, $C$ has length at least 6, and there is an edge, say $x_2 \to y_2$ in the following segment of $C$, which is not contained in such a squat 4-cycle:
\[
\xymatrix@C=0.5em@R=0.5em{
            &y_1 \ar[dr] &            &y_2 \ar[dr] &            &y_3 \ar[dr] & &\cdots &\ar[dr] &    \\
x_1 \ar[ur] &            &x_2 \ar[ur] &            &x_3 \ar[ur] &            & &\cdots &        &x_1
}
\]
Now, if we restrict the auxiliary digraph to the vertices of the pair of adjacent levels containing $C$, we simply have an orientation of the complete bipartite graph on these two levels. In particular, every edge between the $x_i$ vertices and the $y_i$ vertices is present in one direction or the other. Since the edge $x_2 \to y_2$ is not contained in a squat 4-cycle, we can deduce the direction of some of these edges: $x_1 \to y_2$ and $x_2 \to y_3$, among others. Thus, we can write the vertex set of $C$ as the connected union of two shorter squat cycles
\[
\vcenter{\xymatrix@C=0.5em@R=0.5em{
            &y_2 \ar[dr] &            &y_3 \ar[dr] & &\cdots &\ar[dr] &     \\
x_1 \ar[ur] &            &x_3 \ar[ur] &            & &\cdots &        & x_1
}}
\quad\text{and}\quad
\vcenter{\xymatrix@C=0.5em@R=0.5em{
            &y_1 \ar[dr] &            &y_3 \ar[dr] & &\cdots &\ar[dr] &     \\
x_1 \ar[ur] &            &x_2 \ar[ur] &            & &\cdots &        & x_1
}}.
\]
By induction, each of these cycles is generated by squat 4-cycles, so the original cycle $C$ is generated by squat 4-cycles as well.
\end{proof}

\begin{lemma}\label{lem:no-squat-union}
  If two squat 4-cycles in the auxiliary digraph intersect, then they are both contained in the same pair of adjacent levels.
\end{lemma}

\begin{proof}
  Suppose on the contrary that the squat 4-cycle $C$ contained in levels $i, i+1$ intersects the squat 4-cycle $D$ contained in levels $i+1, i+2$. By case analysis, we will use this to find a tall cycle, contradicting \autoref{lem:no-tall}. If $C$ and $D$ intersect in two vertices $y_1, y_2$, then the situation is
\begin{align*}
C\mathpunct{:} \vcenter{\xymatrix@C=0.5em@R=0.5em{
              &             &             & \mbox{\phantom{$z_0$}} \\
              & y_1 \ar[dr] &             & y_2 \ar[dlll]          \\
  x_1 \ar[ur] &             & x_2 \ar[ur] &
}} &&
D\mathpunct{:} \vcenter{\xymatrix@C=0.5em@R=0.5em{
              & z_1 \ar[dr] &             & z_2 \ar[dlll] \\
  y_1 \ar[ur] &             & y_2 \ar[ur] &               \\
              &             &             & \mbox{\phantom{$x_0$}}
}} &&
C \cup D\mathpunct{:} \vcenter{\xymatrix@C=0.5em@R=0.5em{
              &                     & z_1 \ar[dr] &                       & z_2 \ar[dlll] \\
              & y_1 \ar[ur] \ar[dr] &             & y_2 \ar[ur] \ar[dlll] &               \\
  x_1 \ar[ur] &                     & x_2 \ar[ur] &                       &
}}
\end{align*}
  in which case $x_1 \to y_1 \to z_1 \to y_2 \to x_1$ is a tall
  cycle. Otherwise, $C$ and $D$ intersect in a single vertex $y_2$,
  and the situation is
\begin{align*}
C\mathpunct{:} \vcenter{\xymatrix@C=0.5em@R=0.5em{
              &             &             & \mbox{\phantom{$z_0$}} \\
              & y_2 \ar[dr] &             & y_3 \ar[dlll]          \\
  x_2 \ar[ur] &             & x_3 \ar[ur] &
}} &&
D\mathpunct{:} \vcenter{\xymatrix@C=0.5em@R=0.5em{
              & z_1 \ar[dr] &             & z_2 \ar[dlll] \\
  y_1 \ar[ur] &             & y_2 \ar[ur] &               \\
              &             &             & \mbox{\phantom{$x_0$}}
}} &&
C \cup D\mathpunct{:} \vcenter{\xymatrix@C=0.5em@R=0.5em{
              & z_1 \ar[dr] &                     & z_2 \ar[dlll] &               \\
  y_1 \ar[ur] &             & y_2 \ar[ur] \ar[dr] &               & y_3 \ar[dlll] \\
              & x_2 \ar[ur] &                     & x_3 \ar[ur]   &
}}
\end{align*}
  Since $x_2$ and $y_1$ are on adjacent levels, the auxiliary digraph contains the edge between them in one direction or the other. Depending on whether $x_2 \to y_1$ or $y_1 \to x_2$, we have one of the tall cycles
\[
\vcenter{\xymatrix@C=0.5em@R=0.5em{ 
              & z_1 \ar[dr] &             &             &               \\
  y_1 \ar[ur] &             & y_2 \ar[dr] &             & y_3 \ar[dlll] \\
              & x_2 \ar[ul] &             & x_3 \ar[ur] &
}}
\qquad\text{or}\qquad
\vcenter{\xymatrix@C=0.5em@R=0.5em{ 
              &             &             & z_2 \ar[dlll] \\
  y_1 \ar[dr] &             & y_2 \ar[ur] &               \\
              & x_2 \ar[ur] &             &
}}.\qedhere
\]
\end{proof}

\subsection{The canonical partition}

With the help of the cycle lemmas from \hyperref[sec:cycles]{the previous subsection}, we can now define the tangles and clone sets which form the canonical partition.

\begin{definition}[tangle, clone set, canonical partition]\label{def:tangle}\label{def:clone-set}
  The vertex sets of the non-trivial strongly connected components of the auxiliary digraph are called \emph{tangles}.
  If two vertices are not in any tangles and they have the same in- and out-neighbourhoods, they are said to be \emph{clones} of each other.
  The equivalence classes for the relationship of being clones are called \emph{clone sets}.
  Together, the tangles and the clone sets are the blocks of a partition of the set $V$ of vertices called the \emph{canonical partition} and denoted by $B$.
\end{definition}

\begin{figure}[b]
\includegraphics{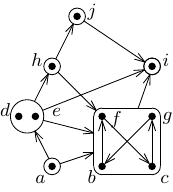}
\qquad \quad
\includegraphics{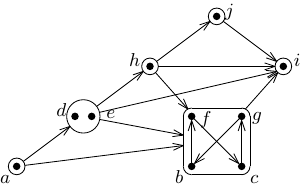}
\caption{Left: the canonical partition with induced edges of the poset
  $P$ from \autoref{fig:poset-auxgraph}. Tangles are enclosed in boxes and clones are
  enclosed in circles. Right: the skeleton with left-right
  ordering of the blocks obtained from the auxiliary digraph from
  \autoref{fig:poset-auxgraph} after applying \autoref{cons:A-to-S}.}
\label{fig:canonicalptn-skeleton}
\end{figure}

\begin{example}
The canonical partition of the poset in \autoref{fig:poset-auxgraph} consists of a tangle with vertices $b,c,f,g$ and five clone sets: the pair $\{d,e\}$ and the singletons $\{a\}$, $\{h\}$, $\{i\}$, and $\{j\}$. See \autoref{fig:canonicalptn-skeleton}.
\end{example}

Informally, the canonical partition is a way to separate out the local structure and the global structure, so to speak, between the vertices of the auxiliary digraph (or equivalently, of the original \tpo-free poset). By `local structure', we mean the relationships between vertices inside each tangle, or each clone set; by `global structure', we mean the relationships between vertices in different blocks of the canonical partition.

\begin{remark}
  The notion of clones is related to the notion of \emph{trimming} of Lewis and Zhang~\cite{LZ}. Also, Zhang~\cite{YXZ} has used techniques involving clones and \tpt-avoidance to prove enumeration results about families of strongly graded posets.
\end{remark}

\begin{remark}\label{rem:altitude}
  The canonical partition of a \tpo-free poset $P$ into clone sets and tangles generalizes the decomposition considered by Skandera and Reed~\cite{SkR} of a \tpo-and-\tpt-free poset given by the \emph{altitude} of the vertices, since the altitude is constant on each clone set, and different clone sets are at different altitudes.

  The altitude of each vertex is still well-defined for \tpo-free posets which \emph{do} contain an induced \tpt{} subposet, and the partition of the vertices according to their altitude gives a finer decomposition than the canonical partition defined here. However, as the example in \autoref{fig:tau} shows, the altitude partition is too fine for \autoref{thm:automorphism} to hold (with the canonical partition replaced by the altitude partition).
  Namely, there is an automorphism $\tau$ which swaps the two vertices with altitude $-1$, the two vertices with altitude $-2$, and two of the three vertices with altitude $2$, as illustrated; but there is no automorphism of the poset which acts non-trivially on a single block of the altitude partition.

  In contrast, for the canonical partition, every automorphism of the poset can be factored as a product of automorphisms which only act non-trivially on a single block.
\end{remark}

\begin{figure}[h]
  \begin{center}
    \includegraphics[scale=.8]{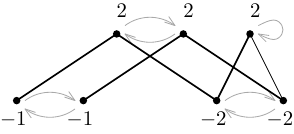}
    \caption{A poset consisting of a single tangle. The vertices are labelled by their altitude~\cite{SkR}, and the arrows illustrate the automorphism $\tau$ mentioned in \autoref{rem:altitude}.}
    \label{fig:tau}
  \end{center}
\end{figure}

\subsection{Skeleta}\label{sec:structure-skeleta}

The structure inside a given clone set is particularly trivial. Two vertices that are clones of each other are completely interchangeable (as made precise in \autoref{thm:automorphism}); they have the same in- and out-neighbourhoods in the auxiliary digraph, they are on the same level, and they are necessarily incomparable in the associated poset.

The structure inside a tangle is richer, but still easy to describe: it consists of a strongly connected orientation of the complete bipartite graph between the vertices on its lower level and the vertices on its upper level.

As we will now see, the structure between the blocks of the canonical partition can be described by a dependence graph, in the sense of the theory of combinatorial traces~\cite{Traces}. Being an acyclic graph, the dependence graph can be seen as a separate poset structure on top of the original poset. To distinguish between the two, we will use terms associated with the left/right axis when discussing the dependence graph, and reserve the traditional up/down axis for the original poset.

\begin{definition}[left, right for vertices]
  Given two vertices $x, y \in V$, we say that $x$ is \emph{left} of $y$ (or $y$ is \emph{right} of $x$), written $x \pathto y$, if there is a path from $x$ to $y$ in the auxiliary digraph.
\end{definition}

\begin{definition}[left, right for blocks]\label{def:lr}
  Given two \emph{distinct} blocks $X, Y \in B$ of the canonical partition we say that $X$ is \emph{left} of $Y$ (or $Y$ is \emph{right} of $X$), written $X \pathto Y$, if there is a path from some (or equivalently all) $x \in X$ to some (or equivalently all) $y \in Y$ in the auxiliary digraph.
\end{definition}

\begin{example}
  In \autoref{fig:canonicalptn-skeleton}, we have the path $h \to j \to i$ from the vertex $h$ to the vertex $i$, so $h$ is left of $i$, denoted $h \pathto i$. We have a path $d \to b \to f \to i$ from a vertex in the block $X = \{d, e\}$ to a vertex in the block $Y = \{i\}$, so $X$ is left of $Y$, denoted $X \pathto Y$.
\end{example}

\begin{definition}[dependence alphabet]\label{def:dep-alpha}
  Let $\Gamma = (\Sigma, D)$ be the \emph{dependence alphabet} which consists of the countable alphabet of symbols
  \[
    \Sigma = \{c_i \mid i \in \NN\} \cup \{t_{i\,i+1} \mid i \in \NN\}
  \]
  and the dependence relation
  \[
    D = \bigcup_{i\in\NN} \{\;t_{i-1\,i}\;, \;c_i\;, \;t_{i\,i+1}\;, \;c_{i+1}\;, \;t_{i+1\,i+2}\;\}^2 \subset \Sigma^2,
  \]
  where the symbol $t_{-1\,0}$ is ignored by convention. The letter $c_i$ will typically denote a clone set on level $i$, and the letter $t_{i\,i+1}$ will typically denote a tangle on levels $i$ and $i+1$.
\end{definition}

\begin{definition}[skeleton]\label{def:skel}
  A \emph{skeleton} $S = (V, B, \ell', E')$ consists of a (finite) set $V$ of vertices, a partition $B$ of $V$ into blocks, a labelling function $\ell' \colon B \to \Sigma$ (where $\Gamma = (\Sigma, D)$ is the dependence alphabet from \autoref{def:dep-alpha}), and a set $E' \subseteq B^2$ of directed edges, denoted $X \to Y$ if $(X, Y) \in E'$, such that:
  \begin{enumerate}[({S}1)]
    \item\label{item:S-edges} for blocks $X, Y \in B$, there is a dependence $(\ell'(X), \ell'(Y)) \in D$ between their labels if and only if either $X = Y$ or $X \to Y$ or $Y \to X$;
    \item\label{item:S-acyclic} the directed graph $(B, E')$ is acyclic;
    \item\label{item:S-source} the directed graph $(B, E')$ is either empty or has a single source (block with no inbound edges), and it is labelled either $c_0$ or $t_{01}$; and
    \item\label{item:S-indirect} if two blocks $X, Y \in B$ such that $X \to Y$ are labelled $\ell'(X) = \ell'(Y) = c_i$ for some $i$, then there exists a third block $Z \in B$ such that $X \to Z \to Y$.
  \end{enumerate}
\end{definition}

\begin{remark}\label{rem:dependence}
  The first two conditions make the skeleton into a dependence graph over $\Gamma$ in the sense of Lemma~2.4.1 of~\cite{Traces}.
\end{remark}

\begin{example}
  The right side of \autoref{fig:canonicalptn-skeleton} gives an example of a skeleton. Note that the skeleton may include edges which are not in the auxiliary digraph, such as $\{h\} \to \{i\}$; these extra edges join every pair of clone sets on the same level.
\end{example}

\begin{construction}\label{cons:A-to-S}
  Given an auxiliary digraph $A = (V, \ell, E)$, we can construct a skeleton $S = (V, B, \ell', E')$ as follows:
  \begin{enumerate}[i., font=\itshape]
    \item keep the same vertex set $V$;
    \item take $B$ to be the canonical partition of $A$;
    \item set $\ell'(X) = c_i$ if $X$ is a clone set on level $i$, or $\ell'(X) = t_{i\,i+1}$ if $X$ is a tangle on levels $i$ and $i+1$;
    \item for blocks $X, Y \in B$, let $X \to Y$ if there is a dependence $(\ell'(X), \ell'(Y)) \in D$ between their labels and $X \pathto Y$.
  \end{enumerate}
\end{construction}

\begin{proof}
  Let us check that this construction yields a skeleton which satisfies all the conditions of \autoref{def:skel}.

  Clearly, the vertex set $V$, the partition $B$, the labelling function $\ell'$, and the set of edges $E'$ produced by the construction are of the right type.

  To check Condition~\ref{item:S-edges}, consider blocks $X, Y \in B$ with $(\ell'(X), \ell'(Y)) \in D$ and $X \neq Y$. By the definition of the dependence relation $D$, the two labels $\ell'(X)$, $\ell'(Y)$ are both in a set of the form
  \[
    \{\;t_{i-1\,i}\;, \;c_i\;, \;t_{i\,i+1}\;, \;c_{i+1}\;, \;t_{i+1\,i+2}\;\}
  \]
  for some $i$. Unless $\ell'(X) = \ell'(Y) = c_i$ or $\ell'(X) = \ell'(Y) = c_{i+1}$, this means that there is a vertex $x \in X$ such that there is a vertex $y \in Y$ with $\ell(y) = \ell(x) \pm 1$, and it follows that $x \to y$ or $y \to x$, so $X \pathto Y$ or $Y \pathto X$, so $X \to Y$ or $Y \to X$.

  The remaining case is when $X$ and $Y$ are two distinct clone sets on the same level, say $i$. Let $x \in X$ and $y \in Y$. The neighbours of $x$ and of $y$ are exactly the vertices on levels $i+1$ and $i-1$, but $x$ and $y$ cannot have the same in- and out-neighbours, since they are in distinct clone sets. Thus, there must be some vertex $z$ on level $i \pm 1$ such that $x \to z \to y$ or $y \to z \to x$. It follows that $X \pathto Y$ or $Y \pathto X$, so $X \to Y$ or $Y \to X$.

  To check Condition~\ref{item:S-indirect}, take $Z$ to be the block containing $z$ in the previous paragraph.

  To check Condition~\ref{item:S-acyclic}, note that every cycle in the auxiliary digraph $A$ is contained within a single tangle, since the tangles are the non-trivial strongly connected components of $A$. It follows that the edges in $E'$ do not form any directed cycles.

  To check Condition~\ref{item:S-source}, consider a block $X \in B$ with a label other than $c_0$ or $t_{01}$, say $c_i$ or $t_{i\,i+1}$ for some $i > 0$. Then, there is some vertex $x \in X$ with $\ell(x) = i > 0$, so there is some vertex $y \in V$ with $\ell(y) = \ell(x) - 1$ and $y \to x$. This vertex $y$ is in a different block $Y$ with $Y \to X$, so $X$ is not a source in the directed graph $(B, E')$.
  Thus, every source of this digraph must have a label in $\{c_0, t_{01}\}$.
  Since it is acyclic, the digraph must contain at least one source, say $X \in B$.
  If another block $Y \in B$ has a label in $\{c_0, t_{01}\}$, then either $X \to Y$ or $Y \to X$ by Condition~\ref{item:S-edges}, so $X \to Y$, and $Y$ is not a source.
  Thus, the source $X$ is unique.
\end{proof}

Before giving the inverse construction, we need a more formal definition of `the structure inside a tangle'.

\begin{definition}[tangle, by itself]
  A \emph{tangle} $T_X = (L_X, U_X, E_X)$ consists of two nonempty sets of vertices $L_X$ and $U_X$, called its \emph{lower level} and \emph{upper level}, respectively, together with a set of directed edges $E_X$ given by a strongly connected orientation of the complete bipartite graph on $L_X$ and $U_X$.
\end{definition}

\begin{example}
  The unique tangle from \autoref{fig:canonicalptn-skeleton} has lower level $L_X = \{b, c\}$, upper level $U_X = \{f, g\}$, and edges $b \to f \to c \to g \to b$.
\end{example}

Since the skeleton is supposed to capture the structure in the auxiliary digraph \emph{between} the components of the canonical partition, and the tangles are supposed to capture the structure \emph{inside} the non-trivial components of the canonical partition, we should be able to recover the auxiliary digraph from the combination of all this data. This is made more formal in the following construction and proposition.

\begin{construction}\label{cons:ST-to-A}
  Given a skeleton $S = (V, B, \ell', E')$ and a collection of tangles $T_X = (L_X, U_X, E_X)$ on the vertex sets of the blocks $X \in B$ labelled $t_{i\,i+1}$ for all $i$, we can construct an auxiliary digraph $A = (V, \ell, E)$ as follows:
  \begin{enumerate}[i., font=\itshape]
    \item keep the same vertex set $V$;
    \item set $\ell(x) = i$ if the vertex $x$ is in a block $X$ labelled $\ell'(X) = c_i$, or on the lower level $L_X$ of a tangle labelled $\ell'(X) = t_{i\,i+1}$, or on the upper level $U_X$ of a tangle $X$ labelled $\ell'(X) = t_{i-1\,i}$;
    \item take $E$ to be the union of the edge sets $E_X$ from the given tangles, together with the edges $x \to y$ for vertices $x \in X \in B$ and $y \in Y \in B$ belonging to distinct blocks such that $\ell(y) = \ell(x) \pm 1$ and $X \to Y$.
  \end{enumerate}
\end{construction}

\begin{proof}
  Let us check that this construction produces an auxiliary digraph which satisfies the conditions of \autoref{def:auxiliary}.

  The constructed sets $V$ of vertices and $E$ of edges are clearly of
  the correct type, and the constructed function $\ell \colon V \to \NN$ is a level function if Condition~\ref{item:a-initial} holds.

  To check Condition~\ref{item:a-edges}, note that the underlying undirected graph for each tangle $T_X$ is a complete bipartite graph between its lower level and its upper level, and that for any two blocks $X, Y \in B$ with vertices on adjacent levels, we have either $X \to Y$ or $Y \to X$.

  To check Condition~\ref{item:a-initial}, consider a vertex $x \in V$ with $\ell(x) > 0$. If $x$ is on the upper level of a tangle, then it is part of some cycle contained in this tangle, and the previous vertex $y$ in this cycle satisfies $\ell(y) = \ell(x) - 1$ and $y \to x$. Otherwise, the vertex $x$ is on some level $i > 0$ and is in a block $X$ labelled $c_i$ or $t_{i\,i+1}$. Since the only source of the skeleton digraph $(B, E')$ has label $c_0$ or $t_{01}$, there must be a block $Y$ with $Y \to X$ with a label in
  \[
    \{\;t_{i-2\,i-1}\;, \;c_{i-1}\;, \;t_{i-1\,i}\;\},
  \]
  otherwise $X$ would not be reachable from the source. This block $Y$ contains a vertex $y$ on level $i-1$, and by construction we have $\ell(y) = \ell(x) - 1$ and $y \to x$.

  To check Condition~\ref{item:a-notall}, note that any cycle of the constructed auxiliary digraph must be contained in a single tangle $T_X$, since the skeleton digraph $(B, E')$ is acyclic, so any 4-cycle must be squat, not tall.
\end{proof}

\begin{definition}[bare, fleshed out skeleton]
  We may refer to the data for \autoref{cons:ST-to-A} as a \emph{fleshed out} skeleton. By contrast, a skeleton by itself could be referred to as a \emph{bare} skeleton.
\end{definition}

\begin{proposition}\label{prop:bijection2}
  \autoref{cons:A-to-S} and \autoref{cons:ST-to-A} are inverses of each other. Hence, they establish a bijection between auxiliary digraphs and fleshed out skeleta.
\end{proposition}

\begin{proof}
  For the first direction (applying \autoref{cons:A-to-S} first, then \autoref{cons:ST-to-A}), we need to show that the level function $\ell \colon V \to \NN$ and the set of edges $E$ is preserved.

  The data of the level function $\ell \colon V \to \NN$ can be recovered almost completely from the composition of the natural projection from the vertex set $V$ to the canonical partition $B$ and the constructed labelling function $\ell' \colon B \to \Sigma$; every vertex belonging to a clone set labelled $c_i$ must be on level $i$, and every vertex belonging to a tangle labelled $t_{i\,i+1}$ must be on level $i$ or $i+1$. This last ambiguity can be resolved by looking at the tangle data, which includes the information of whether each vertex is on the lower level or the upper level of the tangle.

  The edges in $E$ are of the form $x \to y$ or $y \to x$, where the vertex $x \in X$ is on some level $i$, and the vertex $y \in Y$ is on the next level $i+1$. If $X = Y$, then this block is a tangle, so the edge between $x$ and $y$ is recorded in the tangle data $T_X$. Otherwise, the labels $\ell'(X)$ and $\ell'(Y)$ must both be in the set
  \[
    \{\;t_{i-1\,i}\;, \;c_i\;, \;t_{i\,i+1}\;, \;c_{i+1}\;, \;t_{i+1\,i+2}\;\},
  \]
  so the constructed skeleton records the edge between the vertices $x$ and $y$ as part of the edge between the blocks $X$ and $Y$.

  For the other direction (applying \autoref{cons:ST-to-A} first, then \autoref{cons:A-to-S}), we need to show that the partition $B$, the labelling function $\ell'$, and the edge set $E'$ are all preserved. However, if the partition $B$ is preserved, then it should be clear that the data for $\ell'$ and $E'$ is preserved in the form of the data for $\ell$ and $E$ in the auxiliary digraph.

  Let $S_0 = (V, B_0, \ell'_0, E'_0)$ be the original skeleton, $A = (V, \ell, E)$ be the constructed auxiliary digraph, and $S_1 = (V, B_1, \ell'_1, E'_1)$ be the constructed skeleton.
  By Condition~\ref{item:S-acyclic}, the skeleton digraph $(B_0, E'_0)$ is acyclic, so it follows that all the cycles in the auxiliary digraph come from the tangles $T_X$. Furthermore, these tangles are by definition strongly connected, so they form the non-trivial strongly connected components of $A$, and they give the tangles in the canonical partition of the constructed auxiliary digraph. Thus, the blocks of $B_0$ labelled as $t_{i\,i+1}$ become blocks of $B_1$ with the same label.

  It remains to check the clone sets. Let $X \in B_0$ have label $\ell'(X) = c_i$. Then, by elimination, the vertices in $X$ are not part of tangles in the auxiliary digraph $A$; and by construction, they have the same in- and out-neighbourhoods, so they are part of the same clone set in $A$. If the vertices of $X$ form the totality of this clone set in $A$, then we will have $X \in B_1$ as desired.

  Let $Y \in B_0$ be another block with label $\ell'(Y) = c_i$, and without loss of generality, let $X \to Y$ in $S_0$. By Condition~\ref{item:S-indirect}, there is another block $Z \in B_0$ with $X \to Z \to Y$. In fact, by repeatedly looking between $X$ and $Z$ if $\ell'(Z) = c_i$, we can find another $Z$ such that $\ell'(Z) \neq c_i$. Then, the block $Z$ contains a vertex $z$ on level $i \pm 1$, and for any $x \in X$ and $y \in Y$, we have $x \to z \to y$ in the auxiliary digraph $A$. Thus, $x$ and $y$ do not have the same in- and out-neighbourhoods, and they are not part of the same clone set in $A$.
\end{proof}

By putting together all the definitions and constructions of this section so far, we get our first main theorem.

\begin{theorem}\label{thm:bijection}
  There is a bijection between \tpo-free posets and fleshed out skeleta.
\end{theorem}

\begin{proof}
  Apply \autoref{cons:P-to-A} and \autoref{cons:A-to-S} to get a fleshed out skeleton from a \tpo-free poset; and apply \autoref{cons:ST-to-A} and \autoref{cons:A-to-P} to get a \tpo-free poset from a fleshed out skeleton. By \autoref{prop:bijection1} and \autoref{prop:bijection2}, these composite constructions are bijective.
\end{proof}

\subsection{Automorphisms}

Having separated the structure of \tpo-free posets into a global part (the skeleton) and some local parts (the tangles and clone sets), we can now describe the automorphisms of \tpo-free posets; as the following result shows, the global structure is completely rigid, whereas the local structures are completely decoupled.

\begin{definition}[automorphism]
  Let $P = (V, <)$ be a \tpo-free poset and $B$ be its canonical partition into clone sets and tangles.

  A \emph{poset automorphism} of $P$ is a permutation $\sigma$ of $V$ for which $x < y$ iff $\sigma(x) < \sigma(y)$ whenever $x, y \in P$.

  A \emph{clone set automorphism} of a clone set $X \in B$ is any permutation of $V$ which fixes $V \setminus X$ pointwise.

  A \emph{tangle automorphism} of a tangle $X \in B$ with lower level $L$, upper level $U$ and edges $E \subseteq L \times U$ is a permutation $\sigma$ of $V$ which fixes $V \setminus X$ pointwise, and fixes $L$, $U$ and $E$ setwise.

  We write $\aut(P)$ for the group of all poset automorphisms of $P$, and $\aut(X)$ for the group of all clone set or tangle automorphisms of a block $X \in B$.
\end{definition}

\begin{example}
  The canonical partition of the poset $P$ from \autoref{fig:poset-auxgraph} is given in \autoref{fig:canonicalptn-skeleton}. The clone set $\{d, e\}$ has one non-trivial clone set automorphism $\sigma$, which exchanges the vertices $d \leftrightarrow e$ and fixes each of the remaining vertices, $\{a, b, c, f, g, h, i, j\}$. The tangle $\{b, c, f, g\}$ has one non-trivial tangle automorphism $\rho$, which exchanges the lower vertices $b \leftrightarrow c$, exchanges the upper vertices $f \leftrightarrow g$, and fixes each of the remaining vertices, $\{a, d, e, h, i, j\}$. These two automorphisms commute, since they have disjoint supports, and they generate the group $\{\id, \sigma, \rho, \sigma \circ \rho\}$ of all poset automorphisms of $P$.
\end{example}

\begin{theorem}\label{thm:automorphism}
  Let $P = (V, <)$ be a \tpo-free poset and $B$ be its canonical partition into clone sets and tangles. We have the decomposition
  \[
    \aut(P) = \prod_{X \in B} \aut(X),
  \]
  where the product is an internal direct product of subgroups.
\end{theorem}

\begin{proof}
  Since we are dealing with groups of permutations of $V$ and the groups $\aut(X)$ all act on disjoint subsets of $V$, it suffices to show that $\aut(X) \subseteq \aut(P)$ for each $X \in B$ and $\aut(P) \subseteq \prod_{X \in B} \aut(X)$.

  Suppose $X \in B$ is a clone set, and let $\sigma \in \aut(X)$. Since they form a clone set, the vertices of $X$ are on the same level and have the same in- and out-neighbourhoods in the auxiliary digraph $A = (V, \ell, E)$ of the poset $P$. Since $\sigma$ fixes $X$ setwise and fixes $V \setminus X$ pointwise, it follows that $\sigma$ preserves the auxiliary digraph. By \autoref{cons:A-to-P}, it follows that $\sigma$ preserves the relation $<$, so $\sigma \in \aut(P)$. Thus, $\aut(X) \subseteq \aut(P)$ when $X$ is a clone set.

  Now suppose $X \in B$ is a tangle, with lower level $L_X$, upper level $U_X$ and edges $E_X \subseteq L_X \times U_X$, and let $\sigma \in \aut(X)$. Then, the in- and out-neighbourhoods in the auxiliary digraph of vertices in $L_X$ only differ by vertices in $U_X$, and vice versa. Since $\sigma$ fixes $L_X$, $U_X$ and $E_X$ setwise and fixes $V \setminus X$ pointwise, it follows that $\sigma$ preserves the auxiliary digraph, hence it preserves the relation $<$. Thus, $\aut(X) \subseteq \aut(P)$ when $X$ is a tangle.

  Finally, suppose, $\sigma \in \aut(P)$. Let $\ell$ be the level function, $A = (V, \ell, E)$ be the auxiliary digraph, $B$ be the canonical partition, and $S = (V, B, \ell', E')$ be the skeleton for $P$. Since $\sigma$ preserves the order relation $<$, we have $\ell(\sigma(x)) = \ell(x)$ for all vertices $x \in V$, and $x \to y$ is an edge in $E$ iff $\sigma(x) \to \sigma(y)$ is an edge; thus, the level function and the auxiliary digraph are preserved by $\sigma$. In particular, $\sigma$ acts on the blocks of the canonical partition, sending each block $X \in B$ to a block $\sigma(X) = \{\sigma(x) \mid x \in X\} \in B$. Clearly, we have $\ell'(\sigma(X)) = \ell'(X)$ and $X \to Y$ is an edge in $E'$ iff $\sigma(X) \to \sigma(Y)$ is an edge, so the skeleton is preserved by $\sigma$. However, the skeleton induces a total ordering on the set of blocks with a given label, so it follows that $\sigma(X) = X$ for each block $X \in B$. For each block $X \in B$, let
  \[
    \sigma_X(x) = \begin{cases}
      \sigma(x) &\text{if $x \in X$,} \\
      x &\text{if $x \in V \setminus X$.}
    \end{cases}
  \]
  Then, we have $\sigma = \prod_{X \in B} \sigma_X$, and since each $\sigma_X$ preserves the structure of $X$ as a tangle or as a clone set, we have $\sigma_X \in \aut(X)$. Thus, $\aut(P) \subseteq \prod_{X \in B} \aut(X)$.
\end{proof}

\section{Enumeration}\label{sec:enumeration}

Using the decomposition from \autoref{sec:structure} for \tpo-free posets into skeleta containing tangles and clone sets, we obtain the following proposition, which is our main enumerative result. It gives generating functions for the number of distinct \tpo-free posets with respect to the number of vertices, which can be either unlabelled or labelled. The formulas use simple ingredients combined in simple ways, with one exception: the generating functions for the number of distinct bicoloured graphs with respect to the number of vertices. In a sense, then, we reduce the problem of counting \tpo-free posets to the problem of counting bicoloured graphs.

\begin{theorem}\label{thm:enum}
  Let
  \[
    \Pu(x)
      = \sum_{n \geq 0} \left(\text{\parbox{.3\linewidth}{
          \# of \tpo-free posets with $n$ unlabelled vertices
        }}\right) x^n
  \]
  be the ordinary generating function for \tpo-free posets with unlabelled vertices, and
  \[
    \Pl(x)
      = \sum_{n \geq 0} \left(\text{\parbox{.3\linewidth}{
          \# of \tpo-free posets with $n$ labelled vertices
        }}\right) \frac{x^n}{n!}
  \]
  be the exponential generating function for \tpo-free posets with labelled vertices. Then, we have
  \begin{align*}
    \Pu(x) &= S(\Cu(x), \Tu(x, x)) \\
    \Pl(x) &= S(\Cl(x), \Tl(x, x)),
  \end{align*}
  where
  \[
    \Cu(x) = \frac{x}{1 - x}, \qquad
    \Cl(x) = e^x - 1
  \]
  are the generating functions for clone sets from \autoref{prop:genfunc-clone},
  \[
    \Tu(x, y) = 1 - x - y - \Bu(x, y)^{-1}, \qquad
    \Tl(x, y) = e^{-x} + e^{-y} - 1 - \Bl(x, y)^{-1}
  \]
  are the generating functions for tangles from \autoref{prop:genfunc-tangle}, and $S(c, t)$ is the ordinary generating function for skeleta from \autoref{prop:genfunc-skeleta}, which is uniquely determined by the equation
  \[
    S(c, t) = 1 + \frac{c}{1 + c} S(c, t)^2 + t S(c, t)^3.
  \]
\end{theorem}

\begin{proof}
  Given \autoref{prop:bijection1} and \autoref{prop:bijection2}, we can count \tpo-free posets by counting fleshed out skeleta. A fleshed out skeleton consists of a bare skeleton together with some tangles and clone sets, and these tangles and clone sets can be chosen completely independently of each other, provided that the number of tangles and the number of clone sets is as specified by the skeleton.

  It follows from standard generating function theory (see~\cite{BLL,EC1,EC2}, for example) that taking the ordinary generating function for bare skeleta with respect to the number of clone sets and the number of tangles, and plugging in the ordinary (or exponential) generating functions for unlabelled (or labelled) clone sets and tangles with respect to the number of vertices yields the ordinary (or exponential) generating function for fleshed out skeleta on unlabelled (or labelled) vertices with respect to the number of vertices.

  The details for obtaining the generating functions for clone sets, tangles and skeleta are given in the following subsections.
\end{proof}

\begin{remark}
  Note that by \autoref{rem:tpt} and \autoref{prop:tpt}, a \tpo-free poset is also \tpt-free iff it contains no tangles. Therefore, we can recover the generating functions for \tpo-and-\tpt-free posets by setting $t = 0$ in the defining equation for $S(c, t)$; indeed, in the unlabelled case, the ordinary generating function for \tpo-and-\tpt-free posets obtained in this way is $C(x) = S(x/(1-x), 0)$, which satisfies the functional equation $C(x) = 1 + xC(x)^2$ for Catalan numbers.

  In contrast, as noted by~\cite{KRem}, the structure theory for \tpt-free posets in terms of \emph{ascent sequences} (see~\cite{B-MCDK}) doesn't highlight induced \tpo{} subposets, and so doesn't allow the recovery of the \tpo-and-\tpt-free case.
\end{remark}

\begin{remark}
Fran\c{c}ois Bergeron has pointed out that the results of this section can be generalized to obtain the cycle index series (see~\cite{BLL}) for the species of \tpo-free posets.
\end{remark}

\subsection{Clone sets}

As noted in \autoref{sec:structure-skeleta}, the structure inside a clone set is trivial, since a clone set is simply a set of incomparable vertices, so there is exactly one possible `clone set structure' on any given (non-empty) set of vertices. For consistency with the rest of our approach, we record this fact as a pair of generating functions.

\begin{proposition}\label{prop:genfunc-clone}
  The ordinary generating function for clone sets with unlabelled vertices is
  \[
    \Cu(x)
      = \sum_{n \geq 1} \left(\text{\parbox{.3\linewidth}{
          \# of clone sets consisting of $n$ unlabelled vertices
        }}\right) x^n
      = \frac{x}{1 - x}.
  \]
  The exponential generating function for clone sets on sets of labelled vertices is
  \[
    \Cl(x)
      = \sum_{n \geq 1} \left(\text{\parbox{.3\linewidth}{
          \# of clone sets consisting of $n$ labelled vertices
        }}\right) \frac{x^n}{n!}
      = e^x - 1.
  \]
\end{proposition}

\begin{remark}
  Since clone sets appear as components in larger structures (namely, skeleta), it is useful to consider only \emph{non-empty} clone sets in the generating functions above, hence the summations over $n \geq 1$ instead of $n \geq 0$.
\end{remark}

\subsection{Tangles}

According to \autoref{def:tangle}, the structure inside a tangle consists of a partition of its vertices into a lower level and an upper level, together with a strongly connected orientation of the complete bicoloured graph between these two levels. Thus, we can count the possible tangles indirectly by considering \emph{all} orientations of complete bicoloured graphs, and then passing to their strongly connected components.

Since the decomposition theory for orientations of complete bicoloured graphs is simple, we obtain a simple relationship between the generating functions for tangles and for orientations of complete bicoloured graphs.

This can also be seen as a restriction\footnote{Modified so that isolated vertices of the Hasse diagram are allowed to be on level 0 or 1, not just 0.} of the decomposition developed in \autoref{sec:structure} to the case of posets of height at most two (that is, with at most two levels), which are automatically \tpo-free, and whose Hasse diagrams are simply bipartite graphs.

\begin{proposition}\label{prop:genfunc-tangle}
  Let
  \[
    \Bu(x, y)
      = \sum_{n, m \geq 0} \left(\text{\parbox{.33\linewidth}{
          \# of bicoloured graphs with $n$ vertices below and $m$ vertices above, unlabelled
        }}\right) x^n y^m
  \]
  be the ordinary generating function for bicoloured graphs with unlabelled vertices, and
  \[
    \Bl(x, y)
      = \sum_{n, m \geq 0} \left(\text{\parbox{.31\linewidth}{
          \# of bicoloured graphs with $n$ vertices below and $m$ vertices above, labelled
        }}\right) \frac{x^n}{n!} \, \frac{y^m}{m!}
      = \sum_{n, m \geq 0} 2^{nm} \, \frac{x^n}{n!} \, \frac{y^m}{m!}
  \]
  be the exponential generating function for bicoloured graphs with labelled vertices. Then, the ordinary generating function for tangles with unlabelled vertices is
  \begin{align*}
    \Tu(x, y)
      &= \sum_{n, m \geq 2} \left(\text{\parbox{.33\linewidth}{
          \# of tangles with $n$ vertices below and $m$ vertices above, unlabelled
        }}\right) x^n y^m \\
      &= 1 - x - y - \Bu(x, y)^{-1},
  \end{align*}
  and the exponential generating function for tangles with labelled vertices is
  \begin{align*}
    \Tl(x, y)
      &= \sum_{n, m \geq 2} \left(\text{\parbox{.33\linewidth}{
          \# of tangles with $n$ vertices below and $m$ vertices above, labelled
        }}\right) \frac{x^n}{n!} \, \frac{y^m}{m!} \\
      &= e^{-x} + e^{-y} - 1 - \Bl(x, y)^{-1}.
  \end{align*}
\end{proposition}

\begin{remark}
  As with clone sets, tangles appear as components in skeleta, so it is useful to consider only non-empty tangles in the generating functions above, hence the summations over $n, m \geq 2$ for $\Tu(x, y)$ and $\Tl(x, y)$. However, we do consider smaller bicoloured graphs, so we have summations over $n, m \geq 0$ for $\Bu(x, y)$ and $\Bl(x, y)$.
\end{remark}

\begin{proof}
  We proceed by defining an analogue of the canonical partition for bicoloured graphs, which leads to a decomposition of bicoloured graphs into tangles and clone sets. This gives an expression for the generating functions for bicoloured graphs in terms of the generating functions for clone sets and tangles, which we can then invert.

  Let $G = (V, \ell, E)$ be a bicoloured graph, so that $V$ is a set of vertices, the function $\ell \colon V \to \{0, 1\}$ gives a colouring of the vertices, and $E \subseteq \ell^{-1}(\{0\}) \times \ell^{-1}(\{1\})$ is the set of edges joining vertices of colour 0 to vertices of colour 1, denoted by $x < y$ if $x, y \in V$ and $(x, y) \in E$. Then, we can view $G$ as a (coloured) poset of height at most two.

  To this graph $G$, we associate a digraph $A = (V, \ell, E')$, with the same vertex set $V$ and colouring function $\ell$, and edge set $E' \subseteq V^2$, denoted by $x \to y$ if $(x, y) \in E'$, defined as follows: for all $x \in \ell^{-1}(\{0\})$ and $y \in \ell^{-1}(\{1\})$, we have $x \to y$ if $x < y$, and $y \to x$ if $x \not< y$. Thus, the digraph $A$ is an orientation of the complete bicoloured graph with vertex set $V$ and colouring function $\ell$.

  Clearly, for a fixed vertex set $V$ and colouring function $\ell$, this association gives a bijection between the set of all bicoloured graphs and the set of all orientations of the complete bicoloured graph.

  Then, we can use \autoref{def:tangle} to obtain a partition $B$ of the vertices of $A$ into tangles (non-trivial strongly connected components) and clone sets (remaining vertices with the same in- and out-neighbourhoods). Also, for any two blocks $X, Y \in B$, we have either $X = Y$, or $X \pathto Y$, or $Y \pathto X$, so there is a natural total ordering
  \[
    X_0 \pathto X_1 \pathto X_2 \pathto \cdots \pathto X_k
  \]
  of the blocks of $B$. Let us label the blocks by $c_0$ or $c_1$ if they are clone sets (according to their level) or $t_{01}$ if they are tangles, so that the list of blocks above can be represented by a word over the alphabet $\{c_0, c_1, t_{01}\}$.

  It can be verified that the set of possible words obtained in this way is exactly the set of words in $c_0$, $c_1$ and $t_{01}$ with no pair of consecutive letters equal to $c_0 c_0$ or $c_1 c_1$. Then, by a standard inclusion-exclusion argument or by considering the regular expression
  \[
    \{\epsilon, c_0\}\{c_1 c_0\}^*\{\epsilon, c_1\}
    \Big( t_{01} \{\epsilon, c_0\}\{c_1 c_0\}^*\{\epsilon, c_1\} \Big)^*,
  \]
  the ordinary generating function for this set of words can be computed as
  \[
    F(x, y, z)
      = \sum_{n, m, k \geq 0} \left(\text{\parbox{.28\linewidth}{
          \# of these words with $n$ occurrences of $c_0$, $m$ occurrences of $c_1$ and $k$ occurrences of $t_{01}$
        }}\right) x^n y^m z^k
      = \frac{1}{1 - \frac{x}{1+x} - \frac{y}{1+y} - z}.
  \]

  Any bicoloured graph can be represented canonically as a word in $c_0$, $c_1$ and $t_{01}$ with no occurrence of $c_0 c_0$ or $c_1 c_1$ together with a clone set for each $c_0$ and $c_1$ and a tangle for each $t_{01}$, so it follows from standard generating function theory, that we have the equations
  \begin{align*}
    \Bu(x, y) &= F\big(\Cu(x), \Cu(y), \Tu(x, y)\big) = \frac{1}{1 - x - y - \Tu(x,y)} \\
    \Bl(x, y) &= F\big(\Cl(x), \Cl(y), \Tl(x, y)\big) = \frac{1}{e^{-x} + e^{-y} - 1 - \Tl(x,y)}.
  \end{align*}
  Solving these equations for $\Tu(x, y)$ and $\Tl(x, y)$ gives the expressions in the statement of the theorem.
\end{proof}

\subsection{Skeleta}

We now turn to the determination of the number of skeleta with a given number of clone sets and a given number of tangles. As with the bicoloured graphs of the previous subsection, it will be convenient to represent skeleta as words, this time over the alphabet
\[
  \Sigma = \{c_i \mid i \in \NN\} \cup \{t_{i\,i+1} \mid i \in \NN\}
\]
of \autoref{def:dep-alpha}. Unlike the case of bicoloured graphs, there is in general more than one natural representative word for a skeleton, so we will need to pick a canonical representative for each skeleton.

Let $S = (V, B, \ell', E')$ be a skeleton. We already have a notion of \emph{left} and \emph{right} for clone sets and tangles (see \autoref{def:lr}) which gives a partial ordering of the blocks of the canonical partition, so we can consider a listing
\[
  X_1, X_2, X_3, \ldots, X_k
\]
of the blocks $X_i \in B$ where $X_i$ appears before $X_j$ whenever $X_i$ is left of $X_j$. This is exactly a linear extension of the left-right partial ordering. If we replace each clone set at level $i$ in this list by the letter $c_i$ and each tangle at levels $i, i+1$ by the letter $t_{i\,i+1}$, then we obtain a possible word in $\Sigma^*$ which represents the skeleton.

\begin{example}\label{ex:skeleta}
  The two representatives in $\Sigma^*$ for the skeleton given in
  \autoref{fig:canonicalptn-skeleton} are $c_0 c_1 c_2 t_{01} c_3 c_2$ and $c_0 c_1 c_2 c_3 t_{01} c_2$.
\end{example}

As noted in \autoref{rem:dependence}, the labelled digraph $(B, \ell', E')$ on the clone sets and tangles of $S$, which also captures the left-right partial ordering, is a special case of a dependence graph. For such a digraph, we can characterize the set of possible words which represent it: they form the \emph{trace} of the dependence graph~\cite[Section~2.3]{Traces}, which is an equivalence class of words under the commutation relations
\begin{alignat*}{2}
  c_i c_j &= c_j c_i, &\qquad&\text{if $\abs{i - j} \geq 2$}, \\
  c_i t_{j\,j+1} &= t_{j\,j+1} c_i, &&\text{if $i \leq j - 2$ or $i \geq j + 3$}, \\
  t_{i\,i+1} t_{j\,j+1} &= t_{j\,j+1} t_{i\,i+1}, &&\text{if $\abs{i - j} \geq 3$},
\end{alignat*}
given by the complement of the dependence relation $D$ of \autoref{def:dep-alpha}.

\begin{definition}[alphabetic ordering]\label{def:ordering}
  The lexicographically maximal representative of a skeleton $S = (V, B, \ell', E')$ is the lexicographically maximal word in the trace of its dependence graph $(B, \ell', E')$, where the ordering on the letters of $\Sigma$ is
  \[
    c_0 < t_{01} < c_1 < t_{12} < c_2 < t_{23} < c_3 < t_{34} < \cdots.
  \]
\end{definition}

\begin{example}
  Of the two skeleton representatives given in \autoref{ex:skeleta}, the lexicographically maximal one is $c_0 c_1 c_2 c_3 t_{01} c_2$.
\end{example}

The following proposition characterizes which words are lexicographically maximal representatives of skeleta.

\begin{proposition}\label{prop:lexmax}
  Let $w \in \Sigma^*$ be a word over the alphabet
  \[
    \Sigma = \{c_i \mid i \in \NN\} \cup \{t_{i\,i+1} \mid i \in \NN\}
  \]
  from \autoref{def:dep-alpha}. Then, $w$ is the lexicographically maximal representative of \emph{some} skeleton iff:
  \begin{enumerate}[({W}1)]
    \item\label{item:W-start} either $w = \epsilon$ is the empty word, or its first letter is $c_0$ or $t_{01}$;
    \item\label{item:W-pairs} every pair of consecutive letters of $w$ is of the form
      \begin{alignat*}{2}
        c_i & c_j &\qquad&\text{for $j \leq i + 1$; or} \\
        c_i & t_{j\,j+1} &&\text{for $j \leq i + 1$; or} \\
        t_{i\,i+1} & c_j &&\text{for $j \leq i + 2$; or} \\
        t_{i\,i+1} & t_{j\,j+1} &&\text{for $j \leq i + 2$; and}
      \end{alignat*}
    \item\label{item:W-avoid} there is no pair of consecutive letters of $w$ of the form $c_i c_i$ for $i \in \NN$.
  \end{enumerate}
\end{proposition}

\begin{proof}
  We have to show how the properties~\ref{item:S-edges}--\ref{item:S-indirect} of \autoref{def:skel} relate to the conditions~\ref{item:W-start}--\ref{item:W-avoid} through the translation between skeleta and their lexicographically maximal representative words.

  Given the data for a skeleton $S = (V, B, \ell', E')$, the vertex-labelled directed graph $g = (B, \ell', E')$ is a dependence graph in the sense of~\cite[Lemma~2.4.1]{Traces} with respect to the dependence alphabet $\Gamma = (\Sigma, D)$ of \autoref{def:dep-alpha}. Then, according to~\cite[Definition~2.3.3]{Traces}, the representatives (not necessarily lexicographically maximal) for $S$ are exactly given by the \emph{topological orderings} of the digraph $g$, that is, the words obtained by repeatedly choosing a source vertex of $g$ (that is, a block $X \in B$ with in-degree zero), recording its label and deleting the vertex, until the empty graph is obtained. So, since properties~\ref{item:S-edges} and~\ref{item:S-acyclic} are the definition of a dependence graph, they are essentially given for free.

  As noted in~\cite[Section~5]{AK}, out of the representative words for a dependence graph, the lexicographically maximal one is exactly the one obtained by always choosing the source vertex with the \emph{largest} label in the procedure for topological ordering.

  From this observation, we can show that a representative is lexicographically maximal iff Condition~\ref{item:W-pairs} holds as follows. Suppose there is a pair of consecutive letters in the lexicographically maximal representative word which are the labels of two blocks $X_i, X_{i+1} \in B$. Then, either there is an edge $X_i \to X_{i+1}$, in which case the labels $\ell'(X_i)$ and $\ell'(X_{i+1})$ are part of the dependence relation $D$; or there is no such edge, in which case both $X_i$ and $X_{i+1}$ are sources at the $i$th step of topological sorting, and we choose $X_i$ because it has the larger label. In either case, Condition~\ref{item:W-pairs} holds. Conversely, consider a non-maximal representative word $w$. It must be obtained by topological sorting where we don't always choose the source vertex with the largest label. If this happens at step $i$, let $X_j \in B$ be the source vertex with maximal label at step $i$. Then, all the source vertices at step $i$ have a smaller label than $X_j$, and none of the vertices $X_i, X_{i+1}, \ldots, X_{j-1}$ have an edge to $X_j$. Given the structure of the dependence relation $D$ as essentially a union of non-nested intervals in the ordering of \autoref{def:ordering}, it follows that none of the labels $\ell'(X_i), \ell'(X_{i+1}), \ldots, \ell'(X_{j-1})$ can be greater than $\ell'(X_j)$. In particular, $X_{j-1}$ has a smaller label than $X_j$ and there is no edge between them, so the labels $\ell'(X_{j-1})$ and $\ell'(X_j)$ are a pair of consecutive letters in $w$ which fail Condition~\ref{item:W-pairs}.

  Given that the words which satisfy Condition~\ref{item:W-pairs} are the lexicographically maximal representatives of dependence graphs (that is, having properties~\ref{item:S-edges} and~\ref{item:S-acyclic}), it is easy to verify that Property~\ref{item:S-source} holds iff Condition~\ref{item:W-start} holds: since the labels $c_0$ and $t_{01}$ are the smallest labels in the ordering of \autoref{def:ordering}, they can only be chosen during the first step of lexicographically maximal topological sorting if the corresponding block $X \in B$ is the \emph{only} source vertex of the dependence graph.

  Finally, we check the equivalence of Property~\ref{item:S-indirect} and Condition~\ref{item:W-avoid}. Suppose Property~\ref{item:S-indirect} holds, and consider two block $X_i, X_j \in B$ which are clone sets on the same level, with an edge $X_i \to X_j$ in the dependence graph $g$. Then, there is some block $Y \in B$ with $X_i \to Y \to X_j$, and during topological sorting, the vertex $X_j$ cannot become a source vertex immediately after deleting $X_i$, since it must still have $Y$ as an in-neighbour; thus, the labels $\ell'(X_i)$ and $\ell'(X_j)$ do not appear consecutively in the representative word, as prescribed by Condition~\ref{item:W-avoid}. Conversely, suppose Property~\ref{item:S-indirect} fails for two clone sets $X_i, X_j \in B$ on the same level with $X_i \to X_j$ in $g$, and there is no intermediate block $Y \in B$ with $X_i \to Y \to X_j$. Then, during topological sorting, after $X_i$ is chosen and deleted, $X_j$ must become a source vertex, and in fact must be the only vertex becoming a source vertex. Since it has the same label as $X_i$, $X_j$ must be picked as the source vertex with the largest label. Thus, the labels $\ell'(X_i)$ and $\ell'(X_j)$ appear consecutively in the representative word, violating Condition~\ref{item:W-avoid}.
\end{proof}

We now use this characterization of lexicographically maximal representatives to count them, and hence count skeleta.

\begin{proposition}\label{prop:genfunc-skeleta}
  Let
  \[
    S(c, t)
      = \sum_{n, m \geq 0} \left(\text{\parbox{.33\linewidth}{
          \# of skeleta with $n$ clone sets and $m$ tangles
        }}\right) c^n t^m
  \]
  be the ordinary generating function for skeleta. Then, $S(c, t)$ is the unique formal power series solution of the equation
  \[
    S(c, t) = 1 + \frac{c}{1 + c} S(c, t)^2 + t S(c, t)^3.
  \]
\end{proposition}

\begin{proof}
  The formal power series equation above can be turned into a recursive definition for the coefficients of $S(c, t)$, so the fact that the equation has a unique solution is clear. Thus, it suffices to show that $S(c, t)$ is indeed the ordinary generating function for skeleta, or equivalently, for lexicographically maximal representatives of skeleta. We proceed by giving a recursive decomposition of the set of lexicographically maximal representatives of skeleta.

  For each $k \in \NN$, let $\skel_k$ be the set of words $w$ over the truncated alphabet
  \[
    \Sigma_k = \{c_i \mid i \geq k\} \cup \{t_{i\,i+1} \mid i \geq k\}
  \]
  such that
  \begin{enumerate}[({W$_k$}1)]
    \item\label{item:Wk-start} either $w = \epsilon$ is the empty word, or its first letter is $c_k$ or $t_{k\,k+1}$;
    \item\label{item:Wk-pairs} every pair of consecutive letters of $w$ is of the form
      \begin{alignat*}{2}
        c_i & c_j &\qquad&\text{for $j \leq i + 1$; or} \\
        c_i & t_{j\,j+1} &&\text{for $j \leq i + 1$; or} \\
        t_{i\,i+1} & c_j &&\text{for $j \leq i + 2$; or} \\
        t_{i\,i+1} & t_{j\,j+1} &&\text{for $j \leq i + 2$; and}
      \end{alignat*}
    \item\label{item:Wk-avoid} there is no pair of consecutive letters of $w$ of the form $c_i c_i$ for $i \in \NN$,
  \end{enumerate}
  so that, by \autoref{prop:lexmax}, $\skel_0$ is the set of lexicographically maximal representatives of skeleta, and $\skel_k$ is a version of $\skel_0$ with all indices shifted up by $k$. Also, let $\skel_{k,c}$ be the set of words in $\skel_k$ which start with the letter $c_k$, and $\skel_{k,t}$ be the set of words in $\skel_k$ which start with the letter $t_{k\,k+1}$. Then, for each $k \in \NN$, we have the set decompositions
  \begin{align}
    \skel_k &= \{\epsilon\} \sqcup \skel_{k,c} \sqcup \skel_{k,t} \label{eq:ssk}\\
    \skel_{k,c} &= \{c_k\} \skel_{k+1} \skel_k \setminus \{c_k\} \{\epsilon\} \skel_{k,c} \label{eq:sskc}\\
    \skel_{k,t} &= \{t_{k\,k+1}\} \skel_{k+2} \skel_{k+1} \skel_k, \label{eq:sskt}
  \end{align}
  which we now justify.

  Equation~\ref{eq:ssk} is simply a rephrasing of Condition~\ref{item:Wk-start}.

  The first term in Equation~\ref{eq:sskc}, $\{c_k\} \skel_{k+1} \skel_k$, accounts for the fact that, according to Condition~\ref{item:Wk-pairs} and the restricted alphabet $\Sigma_k$, the second letter of a word $w$ in $\skel_{k,c}$ can only be $c_{k+1}$ or $t_{k+1\,k+2}$ if it exists; and furthermore, the word $w$ can be uniquely decomposed as $w = c_k w_{k+1} w_k$, where $w_{k+1}$ is a word in $\skel_{k+1}$ and $w_k$ is a word in $\skel_k$ by looking for the first occurrence, if any, of the letters $c_k$ or $t_{k\,k+1}$. The second term, $\{c_k\} \{\epsilon\} \skel_{k,c}$, accounts for Condition~\ref{item:Wk-avoid}, since we have to exclude the word $w = c_k w_{k+1} w_k$ exactly when $w_{k+1} = \epsilon$ and $w_k$ starts with $c_k$ to avoid having consecutive letters equal to $c_k c_k$.

  Equation~\ref{eq:sskt} similarly accounts for the fact that a first letter of $t_{k\,k+1}$ in a word $w \in \skel_{k,t}$ can only be followed by one of $c_{k+1}$, $t_{k+1\,k+2}$, $c_{k+2}$ or $t_{k+2\,k+3}$. Then, $w$ can be decomposed uniquely as $w = t_{k\,k+1} w_{k+2} w_{k+1} w_k$, where $w_{k+2} \in \skel_{k+2}$, $w_{k+1} \in \skel_{k+1}$ and $w_k \in \skel_k$ by looking for the first occurrence, if any, of a letter in $\Sigma_{k+1} \setminus \Sigma_{k+2}$ and then of a letter in $\Sigma_k \setminus \Sigma_{k+1}$.

  Since each $\skel_k$ is a shifted version of $\skel_0$, they all have the same ordinary generating function with respect to number of clone sets and number of tangles (regardless of levels). Thus, the decomposition equations above turn into a system of equations for $S(c, t)$, which can be solved to obtain the stated equation.
\end{proof}

This concludes the proofs of all the ingredients needed for \autoref{thm:enum}.

\begin{figure}[b]
\centering
\fbox{\begin{minipage}{.95\linewidth}
Consider the 26-vertex \tpo-free poset $P$ with 10 blocks whose skeleton is shown below. Only some of the edges between blocks are drawn, with the others implied.
\[\includegraphics[scale=.75]{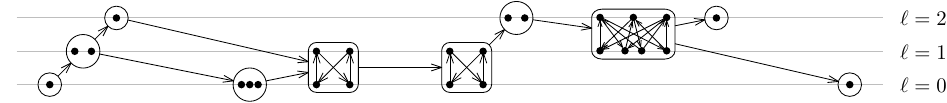}\]
The lexicographically maximal representative for the skeleton of $P$ is \linebreak $w = c_0 c_1 c_2 c_0 t_{01} t_{01} c_2 t_{12} c_2 c_0$, illustrated below.
\[\includegraphics[scale=.75]{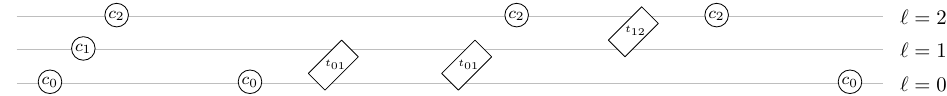}\]
The decorated Dyck path associated with $w$ is below.
\[\includegraphics[scale=.75]{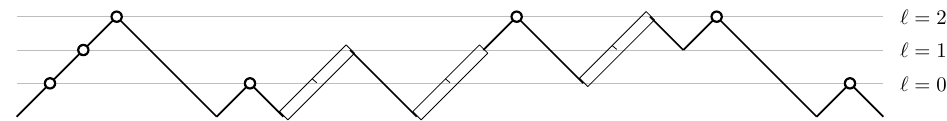}\]
\end{minipage}}
\caption{An example of the decorated Dyck paths discussed in \autoref{rem:dyck}.}
\label{fig:dyck}
\end{figure}

\begin{remark}\label{rem:dyck}
  The lexicographically maximal representative for a skeleton can also be written as a `decorated' Dyck path starting at coordinates $(0, 0)$ in the plane and ending at coordinates $(m, 0)$ for some $m \in \NN$ by replacing each letter $c_i$ with an up step in the direction $(1, 1)$ from $(i, j)$ to $(i+1, j+1)$ for some $j$, replacing each letter $t_{i\,i+1}$ with a double up step in the direction $(2, 2)$ from $(i, j)$ to $(i+2, j+2)$ for some $j$, and filling in the gaps between the resulting segments with down steps in the $(1, -1)$ direction. See \autoref{fig:dyck} for an example.
  When lexicographically maximal representatives are written in this way, the decomposition given in equations~\ref{eq:ssk}--\ref{eq:sskt} corresponds to a natural decomposition of the associated decorated Dyck paths.
\end{remark}

\section{Asymptotics}\label{sec:asympt}

In this section we determine the asymptotics for the number of labelled and unlabelled \tpo-free posets.
Recall that the (univariate) exponential generating
function for
labelled bicoloured graphs is $\Bl(x) = \sum_{n\geq 0} \sum_{i=0}^n
\binom{n}{i} 2^{i(n-i)} \frac{x^n}{n!}$. Let
\[
\bl(n) = [x^n / n!] \, \Bl(x) =
\sum_{i=0}^n \binom{n}{i} 2^{i(n-i)}
\]
be the number of bicoloured graphs on $n$ labelled vertices.
Lewis and Zhang
\cite[Proposition~9.1]{LZ} gave asymptotics for these coefficients.

\begin{proposition}[Lewis and Zhang] \label{prop:estbip}
There exist constants $\alpha_1$ and $\alpha_2$ such that
\[
\bl(2k) \sim \alpha_1 \binom{2k}{k} 2^{k^2}
\qquad \text{and} \qquad
\bl(2k+1) \sim \alpha_2 \binom{2k+1}{k} 2^{k(k+1)}.
\]
\end{proposition}

Recall that the ordinary generating
function for unlabelled bicoloured graphs
up to isomorphism is $\Bu(x)=1+2x+4x^2+8x^3+17x^4 + \cdots$. Let
\[\bu(n) = [x^n] \,
\Bu(x)\]
be the number of such graphs with $n$
vertices. From \cite{JP}, almost all unlabelled bicoloured graphs have a trivial automorphism group, so we can relate the asymptotics of $\bu(n)$ and $\bl(n)$ as follows.

\begin{proposition} \label{prop:labVSunlab}
If $\bu(n)$ is the number of bicoloured
graphs with $n$ unlabelled vertices and $\bl(n)$ is the
number of bicoloured graphs with $n$ labelled vertices then
\[
n! \cdot \bu(n) \sim
\bl(n).
\]
\end{proposition}

Lewis and Zhang~\cite[Theorem~9.2]{LZ} also gave the asymptotics for the number of
(weakly) graded \tpo-free posets with $n$ labelled
vertices\footnote{Recall that a poset $P$ is {\em weakly graded} if there exists a
rank function $\rho: P \to \{0,1,2,\ldots\}$ such that if $a<b$ is a
covering relation then $\rho(b)-\rho(a)=1$. A poset is {\em strongly
  graded} if it is weakly graded, minimal vertices have
the same rank, and maximal vertices have the same rank (\textit{i.e.}, all
maximal chains in the poset have the same number of vertices).}. Using \autoref{prop:labVSunlab} and their method of
proof one gets the asymptotics for the number of (weakly) graded \tpo-free posets with $n$ unlabelled vertices.

\begin{theorem}[Lewis and Zhang]\label{thm:asymptGP}
Let $\pl^{\textnormal{g}}(n)$ and
$\pu^{\textnormal{g}}(n)$ be the number of strongly graded
\tpo-free posets with $n$ labelled vertices and $n$ unlabelled
vertices respectively, and let $\pl^{\textnormal{w}}(n)$ and
$\pu^{\textnormal{w}}(n)$ be the corresponding
numbers for weakly graded posets. Then
\begin{enumerate}[(i),noitemsep]
\item $\pl^{\textnormal{g}}(n) \sim \pl^{\textnormal{w}}(n) \sim \bl(n)$, and
\item $\pu^{\textnormal{g}}(n) \sim \pu^{\textnormal{w}}(n) \sim \bu(n)$.
\end{enumerate}
\end{theorem}

We are ready to state the main result of this section, which gives the
asymptotics for the number of \tpo-free posets with labelled and
unlabelled vertices respectively.

\pagebreak

\begin{theorem} \label{thm:asymptP}
If $\pl(n)$ is the number of \tpo-free posets with
$n$ labelled vertices and $\pu(n)$ is the number of
\tpo-free posets with $n$ unlabelled vertices then
\begin{enumerate}[(i),noitemsep]
\item $\pl(n) \sim \bl(n)$, and
\item $\pu(n) \sim \bu(n)$.
\end{enumerate}
\end{theorem}

Combining theorems~\ref{thm:asymptGP} and~\ref{thm:asymptP}, it
follows that almost all \tpo-free posets are (weakly) graded. This
fact may be surprising at first, but it is actually a consequence of
the stronger fact that almost all \tpo-free posets have Hasse
diagrams which are bicoloured graphs, meaning that they have exactly
two levels.

\begin{remark}
In \autoref{table:asympt} we compare the asymptotics for the number of \tpo-free posets on $n$ unlabelled vertices with the known asymptotics for the number of
\tpo-and-\tpt-free posets, \tpt-free posets~\cite{B-MCDK}, \tpo-free posets, and all posets~\cite{KR,JP} on $n$ unlabelled vertices. The asymptotics
compare in the following way:
\[
	\text{\tpo-and-\tpt-free} \, \ll \, \text{\tpt-free} \, \ll \, \text{\tpo-free} \, \ll \, \text{all posets}.
\]
\end{remark}

\begin{table}[b]
\begin{center}
\begin{tabular}{lll}
\toprule
Class of posets & Asymptotics & Asymptotics (base 2) \\ \midrule
\rule{0pt}{3ex} \tpo-and-\tpt-free & ${}=\frac{1}{n+1}\binom{2n}{n}$ & $2^{2n + O(\log n)}$ \\
\rule{0pt}{3ex} \tpt-free & $\beta_1 n! \sqrt{n} \left(\frac{6}{\pi^2}\right)^n$ \cite{B-MCDK} & $2^{\beta_3 n\log n + O(n)}$ \\
\rule{0pt}{3ex} \tpo-free & $\bl(n) / n!$ & $2^{n^2/4-\beta_3 n\log n + O(n)}$ \\
\rule{0pt}{3ex} all posets & $\frac{\beta_2}{n! \sqrt{n}} 2^{n^2/4 +
3n/2}$ \cite{KR,JP} & $2^{n^2/4-\beta_3 n\log n + O(n)}$ \\ \bottomrule
\end{tabular}
\medskip
\caption{Asymptotic comparison of the numbers of
\tpo-and-\tpt-free posets, \tpt-free
posets, \tpo-free posets, and arbitrary
posets with $n$ unlabelled vertices. Here, $\bl(n)$ is the number of bicoloured graphs on $n$ labelled vertices and $\beta_1, \beta_2, \beta_3$ are constants. The asymptotics for \tpo-free posets follows from \autoref{thm:asymptP},
\autoref{prop:labVSunlab}, and \autoref{prop:estbip}.}
\label{table:asympt}
\end{center}
\end{table}

Like the proof of \autoref{thm:asymptGP}, the proof of \autoref{thm:asymptP} relies on the following result of Bender~\cite[Theorem~1]{B}.

\begin{theorem}[Bender] \label{thm:Bender}
Suppose that $F(x) = \sum_{n\geq 1} f_n x^n$, that $H(x,y)$ is a formal
power series in $x$ and $y$, and that $G(x)=\sum_{n\geq 0} g_n x^n =
H(x,F(x))$. Let $C = \left. \frac{\partial H}{\partial y}
\right|_{(0,0)}$. Suppose that
\begin{enumerate}[1.,noitemsep]
\item[1.] $H(x,y)$ is analytic in a neighbourhood of $(0,0)$,
\item[2.] $\lim_{n\to \infty} \frac{f_{n-1}}{f_n} = 0$, and
\item[3.] $\sum_{k=1}^{n-1} |f_kf_{n-k}| = O(f_{n-1})$.
\end{enumerate}
Then
\[
g_n = C\cdot f_n + O(f_{n-1}),
\]
and in particular $g_n \sim C\cdot f_n$.
\end{theorem}

\begin{proof}[Proof of {\hyperref[thm:asymptP]{Theorem~\ref*{thm:asymptP}(i)}}]
Let $\Hl(x,y)$ be the formal power series in $x$ and $y$ defined by
\begin{equation}\label{eq:Hlbl}
  \Hl(x,y) = S(e^x-1, 2e^{-x} - 1 - (1+y)^{-1}),
\end{equation}
where $S(c,t)$ is the unique formal power series solution of the cubic equation
\begin{equation} \label{eq:cubic}
S(c,t) = 1 + \frac{c}{1+c} S(c,t)^2 + tS(c,t)^3,
\end{equation}
as defined in \autoref{thm:enum}. From Equation~\eqref{ordgs}, we have that
\[
  \Hl(x,\Bl(x)-1)
    = \sum_{n\geq 0} \pl(n) \frac{x^n}{n!}.
\]
In order to apply \autoref{thm:Bender} we first check its three
conditions.
Condition 1, that
$\Hl(x,y)$ be analytic in a neighbourhood of $(0,0)$,
follows from the way it is defined in \eqref{eq:Hlbl}.
Lewis and Zhang~\cite{LZ} verified that the coefficients
$\bl(n) / n!$ of the generating function
$\Bl(x)-1$ satisfy Conditions 2 and 3. That is,
\[
\lim_{n\to \infty}
\frac{n \cdot \bl(n-1)}{\bl(n)} = 0
\]
and
\[
\sum_{k=1}^{n-1} \abs{\frac{\bl(k)}{k!}} \cdot \frac{\bl(n-k)}{(n-k!)} =
O\left(\frac{\bl(n-1)}{(n-1)!}\right).
\]

Using the chain rule on \eqref{eq:Hlbl} and implicit differentiation on \eqref{eq:cubic}, we have
\[
  \frac{\partial}{\partial y} \Hl(x,y)
    = \left.
      \frac{S(c, t)^3}{(1 + y)^2 \left( 1 - \frac{2 c S(c, t)}{1 + c} - 3 t S(c, t)^2 \right)}
    \right|_{\substack{c = e^x - 1 \hfill \\ t = 2e^{-x} - 1 - (1 + y)^{-1} \hfill}}
\]
and at $(x, y) = (0, 0)$, it follows that
\begin{align*}
  C & = \left.\frac{\partial}{\partial y}
    \Hl(x,y)\right|_{(x,y)=(0,0)}\\
  &= S(0,0)^3 = 1,
\end{align*}
where the last equality follows from \eqref{eq:cubic}. 
So, by \autoref{thm:Bender}, we have
\[
\frac{\pl(n)}{n!} = \frac{\bl(n)}{n!} +
  O\left(\frac{\bl(n-1)}{(n-1)!}\right) \sim \frac{\bl(n)}{n!}. \qedhere
\]
\end{proof}

\begin{proof}[Proof of {\hyperref[thm:asymptP]{Theorem~\ref*{thm:asymptP}(ii)}}]
Let $\Hu(x,y)$ be the formal power series in $x$ and $y$ defined by
\begin{equation}\label{eq:Hunl}
  \Hu(x,y) = S(x (1-x)^{-1}, 1-2x-(1+y)^{-1}),
\end{equation}
where $S(c,t)$ is again the formal power series solution of \eqref{eq:cubic} as defined in \autoref{thm:enum}. From Equation~\eqref{expgs}, we have that
\[
  \Hu(x, \Bu(x)-1)
    = \sum_{n\geq 0} \pu(n) x^n.
\]
Again we check the conditions of \autoref{thm:Bender}.
As in the labelled case above, from the definition of $\Hu(x,y)$ in \eqref{eq:Hunl} we
see that it is analytic so Condition 1 holds.
From \autoref{prop:labVSunlab}, we have $\bu(n) \sim
\bl(n) / n!$, and since the coefficients
$\bl(n) / n!$ satisfy Conditions 2 and 3, so do the
coefficients $\bu(n)$.

Using the chain rule on \eqref{eq:Hunl} and implicit differentiation on \eqref{eq:cubic}, we have
\[
  \frac{\partial}{\partial y} \Hu(x,y)
    = \left.
      \frac{S(c, t)^3}{(1 + y)^2 \left( 1 - \frac{2 c S(c, t)}{1 + c} - 3 t S(c, t)^2 \right)}
    \right|_{\substack{c = x (1 - x)^{-1} \hfill \\ t = 1 - 2x - (1 + y)^{-1} \hfill}}
\]
and at $(x, y) = (0, 0)$, it follows that
\[
  C = \left.\frac{\partial}{\partial y} \Hu(x,y)\right|_{(x,y)=(0,0)}
  = S(0,0)^3 = 1.
\]
So, by \autoref{thm:Bender}, we have
\[
\pu(n) = \bu(n) +
  O(\bu(n-1)) \sim \bu(n). \qedhere
\]
\end{proof}

\section{Acknowledgements}\label{sec:thanks}

This work grew out of a working session of the algebraic combinatorics group at LaCIM with active participation from Chris Berg, Franco Saliola, Luis Serrano, and the authors.
It was facilitated by computer exploration using various types of mathematical software, including Sage~\cite{sage} and its algebraic combinatorics features developed by the Sage-Combinat community~\cite{sage-combinat}.
In addition, the authors would like to thank Joel Brewster Lewis for
several conversations and
suggesting looking at asymptotics,
Mark Skandera and Yan X Zhang for helpful discussions,
and the anonymous referee for comments which led to substantial improvements of this paper.


\begin{thebibliography}{10}

\bibitem{AK}
A.~V. Anisimov and D.~E. Knuth.
\newblock Inhomogeneous Sorting.
\newblock {\em  Internat. J. Comput. Inform. Sci.}, 8(4):255--260, 1979.

\bibitem{ASV}
M.~D. Atkinson, Bruce~E. Sagan, and Vincent Vatter.
\newblock Counting \tpo-avoiding permutations.
\newblock {\em European Journal of Combinatorics}, 33(1):49--61, 2011.

\bibitem{BLL}
Fran\c{c}ois Bergeron, Gilbert Labelle, and Pierre Leroux.
\newblock {\em Combinatorial Species and Tree-like Structures}.
\newblock Cambridge University Press, 1997.

\bibitem{B}
Edward A. Bender
\newblock An asymptotic expansion for the coefficients of some
  formal power series.
\newblock {\em J. London Math. Soc.}, 9(2):451--458, 1974/75.


\bibitem{B-MCDK}
Mireille Bousquet-M{\'e}lou, Anders Claesson, Mark Dukes and Sergey Kitaev,
\newblock {\tpt-free posets, ascent sequences and pattern
  avoiding permutations},
\newblock {\em J. of Combin. Theory, Ser. A},
117(7):884--909, 2010.



\bibitem{Traces}
Volker Diekert and Grzegorz Rozenberg.
\newblock {\em The Book of Traces}.
\newblock World Scientific Publishing, Singapore, 1995.


\bibitem{G2}
Vesselin Gasharov.
\newblock Incomparability graphs of \tpo-free posets are $s$-positive.
\newblock {\em Discrete Math.}, 157:211--215, 1996.

\bibitem{GP}
Mathieu Guay-Paquet.
\newblock A modular relation for the chromatic symmetric functions of \tpo-free
posets.
\newblock Preprint, \href{http://arxiv.org/abs/1306.2400}{ArXiv:1306.2400}, 2013.

\bibitem{GMRProc}
Mathieu Guay-Paquet, Alejandro H. Morales and Eric Rowland.
\newblock Structure and enumeration of \tpo-free posets (extended
abstract).
\newblock 25th International Conference on Formal Power Series and Algebraic Combinatorics (FPSAC 2013).
\newblock {\em Discrete Math. Theor. Comput. Sci. Proc.}, 253--264, 2013.

\bibitem{KRem}
Sergey Kitaev and Jeffrey Remmel.
Enumerating \tpt-free posets by the number of minimal elements and
other statistics.
\newblock {\em Discrete Appl. Math.},
159 (2011) 2098--2108.

\bibitem{KR}
Daniel J. Kleitman and Bruce L. Rothschild.
\newblock Asymptotic enumeration of partial orders on a finite set.
\newblock {\em Trans. Amer. Math. Soc.},
205:205--220, 1975.

\bibitem{LZ}
Joel~Brewster Lewis and Yan~X Zhang.
\newblock Enumeration of graded \tpo-avoiding posets.
\newblock {\em J. Combin. Theory Ser. A},
120(6):1305--1327, 2013.

\bibitem{OEIS}
The OEIS Foundation.
\newblock The {O}n-{L}ine {E}ncyclopedia of {I}nteger {S}equences.
\newblock \url{http://oeis.org}.

\bibitem{JP}
Hans J\"urgen Pr\"omel.
\newblock Counting unlabeled structures,
\newblock {\em J. Combin. Theory Ser. A.},
\newblock 44(1):83--93, 1987.

\bibitem{sage-combinat}
The {S}age-{C}ombinat community.
\newblock {S}age-{C}ombinat: enhancing {S}age as a toolbox for computer
  exploration in algebraic combinatorics, 2013.
\newblock \url{http://combinat.sagemath.org}.

\bibitem{Sk1}
Mark Skandera.
\newblock A characterization of \tpo-free posets.
\newblock {\em J. Combin. Theory Ser. A}, 93(2):231--241, 2001.

\bibitem{SkR}
Mark Skandera and Brian Reed.
\newblock Total nonnegativity and \tpo-free posets.
\newblock {\em J. Combin. Theory Ser. A}, 103(2):237--256, 2003.


\bibitem{St1}
Richard~P. Stanley.
\newblock A symmetric function generalization of the chromatic polynomial of a
  graph.
\newblock {\em Advances in Math.}, 111:166--194, 1995.

\bibitem{St2}
Richard~P. Stanley.
\newblock Graph colorings and related symmetric functions: ideas and
  applications: {A} description of results, interesting applications, {\&}
  notable open problems.
\newblock {\em Discrete Math.}, 193:267--286, 1998.

\bibitem{EC1}
Richard~P. Stanley.
\newblock {\em Enumerative Combinatorics, Volume 1}.
\newblock Cambridge University Press, second edition, 2011.

\bibitem{EC2}
Richard~P. Stanley.
\newblock {\em Enumerative Combinatorics, Volume 2}.
\newblock Cambridge University Press, 1999.

\bibitem{StSt}
Richard~P. Stanley and John~R. Stembridge.
\newblock On immanants of {J}acobi-{T}rudi matrices and permutations with
  restricted position.
\newblock {\em J. Combin. Theory Ser. A}, 62(2):261--279, 1993.

\bibitem{sage}
W.~A. Stein et~al.
\newblock Sage mathematics software (version 5.3), 2013.

\bibitem{YXZ}
Yan~X Zhang.
\newblock Variations on graded posets.
\newblock In preparation, 2014.

\end{thebibliography}

\end{document}